\documentclass{article} 

\usepackage[utf8]{inputenc}
\usepackage[T1]{fontenc}
\usepackage{lmodern}
\usepackage{amsmath, amssymb, amsthm, amstext,mathrsfs}
\usepackage{thmtools}
\usepackage{microtype}
\usepackage[english=usenglishmax]{hyphsubst}
\usepackage{esint}
\usepackage{graphicx}
\usepackage[babel,german=quotes]{csquotes}
\usepackage{bibentry}
\usepackage{tabularx}		
\usepackage{commath}			% for \norm{...}
\usepackage{setspace}
\usepackage{bm}
\usepackage{bbm}
\usepackage{threeparttable}
\usepackage{comment}\includecomment{notthisone}
\usepackage{enumerate}
\usepackage{epstopdf}
\usepackage{enumerate}
\usepackage{dsfont}
\usepackage{pstool}
\usepackage{subfig}
\usepackage{natbib}
\usepackage{hyperref}
\usepackage{authblk}
\usepackage{algorithm}
\usepackage{algcompatible}
\usepackage{algorithmicx}
\usepackage{array, booktabs}
\usepackage{setspace}
\usepackage{mathcomp}
\usepackage{multirow}
\usepackage[paper=a4paper,left=25mm,right=25mm,top=25mm,bottom=25mm]{geometry}

\floatname{algorithm}{Algorithm}

\providecommand{\keywords}[1]{\textbf{\textit{Keywords: }} #1}
\providecommand{\classification}[1]{\textit{2020 Mathematics Subject Classification: } #1}
\makeatletter
\def\BState{\State\hskip-\ALG@thistlm}
\makeatother

\usepackage{mathtools}
\mathtoolsset{showonlyrefs}

\makeatletter

\newtheoremstyle{Definition}
  {0.2cm}                   %Space above
  {0.2cm}                   %Space below
  {\normalfont}           %Body font
  {}                      %Indent amount (empty = no indent,
  {\bfseries}  						%Thm head font
  {.}                     %Punctuation after thm head
  { }              				%Space after thm head: " " = normal interword
  {}
\newtheoremstyle{Theorem}
  {0.2cm}                   %Space above
  {0.2cm}                   %Space below
  {\itshape}           		%Body font
  {}                      %Indent amount (empty = no indent,
  {\bfseries}  						%Thm head font
  {.}                     %Punctuation after thm head
  { }              				%Space after thm head: " " = normal interword
  {}
\theoremstyle{Theorem}
	\newtheorem{cor}{Corollary}
	
	\newtheorem{lem}{Lemma}
	\newtheorem{thm}{Theorem}

	\newtheorem{ass}{Assumption}
	\newtheorem{scheme}{Scheme}

\theoremstyle{Definition}
	\newtheorem{example}{Example}
	\newtheorem{rem}{Remark}
	\newtheorem{defn}{Definition}

\newcommand{\upd}{\mathrm{d}}

\allowdisplaybreaks
\makeatother

\begin{document}

%\newpage{}

%\thispagestyle{empty}
%\tableofcontents
%\newpage{}

%\pagenumbering{arabic}

\title{Simulating Continuous-Time Autoregressive Moving Average Processes Driven By $p$-Tempered $\alpha$-Stable L\'evy Processes} 
\author[a]{Till Massing\thanks{Faculty of Economics, University of
Duisburg-Essen, Universit{\"{a}}tsstr.~12, 45117 Essen, Germany.\\E-Mail:
till.massing@uni-due.de}}

\maketitle

\begin{abstract}
We discuss simulation schemes for continuous-time autoregressive moving average (CARMA) processes driven by tempered stable L\'evy noises. CARMA processes are the continuous-time analogue of ARMA processes as well as a generalization of Ornstein-Uhlenbeck processes. However, unlike Ornstein-Uhlenbeck processes with a tempered stable driver (see, e.g., \cite{qu2021}) exact transition probabilities for higher order CARMA processes are not explicitly given. Therefore, we follow the sample path generation method of \cite{Kawai2017} and approximate the driving tempered stable L\'evy process by a truncated series representations. We derive a result of a series representation for $p$-tempered $\alpha$-stable distributions extending \cite{Rosinski2007}. We prove approximation error bounds and conduct Monte Carlo experiments to illustrate the usefulness of the approach.
\end{abstract}

\keywords{Tempered stable distributions, CARMA processes, L\'evy processes, $p$-tempering, simulation, series representation}\\
\classification{60E07, 60G10, 60G51, 62M10}
\microtypesetup{activate=true}

\section{Introduction}\label{sec:intro}

We discuss simulation schemes for continuous-time autoregressive moving average (CARMA) processes driven by tempered $\alpha$-stable L\'evy noises.
CARMA processes are the continuous-time analogue of discrete-time ARMA processes and extend the concept of Ornstein-Uhlenbeck (OU) processes. Ornstein-Uhlenbeck processes, also referred to as CAR(1), represent the continuous-time analogue of autoregressive processes with a lag order of 1. They find extensive use in financial literature, see, e.g., \cite{barndorff2001non}. While originally introduced for Gaussian CARMA processes driven by Brownian motion \cite[]{Brockwell1994}, CARMA processes have been generalized to be driven by arbitrary L\'evy processes \cite[]{Brockwell2001}. Non-negative L\'evy-driven CARMA processes have attracted interest, particularly in the context of modeling stochastic volatility \cite[]{Brockwell2011}. While non-negative OU processes are often employed for this purpose, CARMA processes with higher lag orders offer a more flexible autocorrelation function.

Simulation of L\'evy-driven CARMA processes requires the formulation of both a scheme for simulating trajectories of the background-driving L\'evy process (BDLP) and a strategy for subsequently sampling a path of the stationary CARMA process. Our primary interest lies in achieving high-frequency sampling to obtain a quasi-continuous sample path. In \cite{todorov} and \cite{Kawai2017}, the authors propose series representations of L\'evy processes as fundamental building blocks. Series representations for L\'evy processes trace back to \cite{Rosinski2001} and provide a method to sample a path with an arbitrary number of jumps. In \cite{Kawai2017} a simple scheme for sampling CARMA paths is presented which is based on a series representation for the BDLP, an approach that we follow in this paper.

We here focus on the special case of CARMA models based on tempered stable distributions which have a long history throughout probability theory and applications in finance. These distributions emerge through the tempering of the L\'evy measure of stable distributions with a suitable tempering function. The concept of tempered stable distributions was first introduced by \cite{Koponen1995}, who termed the associated L\'evy process a ``smoothly truncated L\'evy flight'', representing a generalization of Tweedie distributions \cite[]{Tweedie1984}. Subsequently, tempered stable distributions have been generalized in several directions, primarily by expanding the class of admissible tempering functions. The seminal paper \cite{Rosinski2007} presents a general framework for tempered stable distributions and derives a series representation which we here extend. The CGMY distribution \cite[]{CGMY2002} is a well-known special case of the classical tempered stable distribution that was introduced to model log-returns of stock prices. Tempered stable distributions have frequently been used for financial applications \cite[]{Kim2008,rachev2011financial,Fallahgoul2019}.

In this paper, we treat so-called $p$-tempered $\alpha$-stable distributions, as introduced by \cite{Grabchak2012}, where $p>0$ and $\alpha\in(0,2)$. The parameter $p$ governs the extent of tempering, while $\alpha$ represents the stability index of the associated stable distribution. This class of distributions contains, e.g., the standard class of \cite{Rosinski2007} for $p=1$ and the class of \cite{Bianchi2011} for $p=2$. See also \cite{Grabchak2016book} for a survey. $p$-tempered $\alpha$-stable distributions are potential tools for various financial applications. For example, \cite{Sabino2022exact} discussed the pricing of a strip of daily European options in energy or commodity markets and the modeling of future markets. Modeling flight length in mobility models is another domain of application as discussed by \cite{Grabchak2016book}.

\cite{Grabchak2019} develops a fast and easy-to-use rejection sampling algorithm for $p$-tempered $\alpha$-stable distributions. However, there is a limitation that we aim to address in this paper. The rejection algorithm requires $0<\alpha<p$, which leads to the exclusion of important cases. For example, it does not cover instances where $\alpha\in(1,2)$ for $p=1$, which is the standard case presented by \cite{Rosinski2007}. Therefore, we opt for a different approach by deriving a series representation for $p$-tempered $\alpha$-stable distributions that also allows $\alpha>p$. There are a few other series representations available derived by \cite{Rosinski2010} which we briefly review in Section \ref{sec:distributions}.

As previously mentioned, CARMA processes are a natural extension of OU processes. In the literature, simulation algorithms for tempered stable OU processes have been extensively studied. Notably, there exist two distinct versions of these processes with a subtle yet crucial difference. Firstly, TS-OU processes are OU processes having a tempered stable distribution as stationary marginal distribution. Secondly, OU-TS processes are OU processes driven by a tempered stable BDLP. Both variants have been examined in, among others, \cite{Zhang2009}, \cite{KawaiMasuda2011} and \cite{qu2021}. These studies have derived transition probabilities and exact sampling schemes for tempered stable subordinators. Various extensions have since been established. For instance, \cite{Sabino2022OU} and \cite{Grabchak2023} have expanded upon these methodologies to simulate TS-OU and OU-TS processes using classical tempered stable distributions and $p$-tempered stable distributions, respectively. However, the approach of explicitly deriving the transition probabilities is not transferable to general CARMA processes of higher lag order due
to the more complicated structure than OU processes (sum of exponential functions instead of one exponential function), see also \cite{Kawai2017}. Furthermore, considering a TS-CARMA process -- a CARMA process having a stationary TS distribution -- is computationally infeasible. Therefore, we focus on the sampling of CARMA processes driven by a tempered stable L\'evy process, called CARMA-TS processes, specifically utilizing $p$-tempered stable distributions.

The contributions of this paper are the following. Firstly, we establish a novel series representation for $p$-tempered $\alpha$-stable distributions by generalizing the findings of \cite{Rosinski2007}. Based on this, we then derive a series representation and hence a simulation scheme for CARMA processes driven by $p$-tempered $\alpha$-stable L\'evy processes and we study approximation errors.

The paper is organized as follows. In Section \ref{sec:distributions} we define $p$-tempered stable distributions and highlight some important subclasses. Section \ref{sec:carma} presents the basics of CARMA processes and some useful properties. In Section \ref{subsec:ptempseries} we state and prove the main theoretical result of a series representation for $p$-tempered stable L\'evy processes. In Section \ref{subsec:ptempCARMA} we derive the simulation scheme for CARMA-TS processes. Monte Carlo experiments in Section \ref{sec:MCstudy} illustrate the usefulness of the proposed simulation scheme and show its limitations. Section \ref{sec:conclusion} concludes. The proofs are to be found in Appendix \ref{app:proofs}.

\section{Tempered stable distributions}\label{sec:distributions}

%\subsection{Generalities}\label{sec:tempgeneralities}

In this section, we briefly review tempered stable distributions as in \cite{Rosinski2007}. We particularly emphasize the extension of these distributions known as $p$-tempered stable distributions introduced by \cite{Grabchak2012} and reviewed in \cite{Grabchak2016book}. At the end of the section, we provide relevant examples. To start, we recall some relevant results from probability theory. An $\mathbb{R^d}$-valued process $\{L_t\}_{t\in\mathbb{R_+}}$ is called a \emph{L\'{e}vy process} if $X_0=0$ a.s., it has independent and stationary increments, it is stochastically continuous and the path function is c\`{a}dl\`{a}g a.s. A \emph{subordinator} is a one-dimensional, (a.s.)~non-decreasing L\'{e}vy process. We express L\'{e}vy process using the L\'evy-Khintchine representation, i.e., a probability distribution $\mu$ on $\mathbb{R}^d$ of a random variable $X$ is infinitely divisible if and only if there exists a unique triple (called L\'evy triple) $(\gamma,A,M)$, where $M$ is a measure on $\mathbb{R}^d$ called \emph{L\'evy measure}.

\begin{defn}\label{defn:ptemp}
Let $\alpha\in(0,2)$ and $p>0$. An infinitely divisible probability measure $\mu$ on $\mathbb{R}^d$ is called $p$-\emph{tempered} $\alpha$-\emph{stable distribution} if it has no Gaussian part and its L\'evy measure has the form
\begin{equation}
\label{eq:Levymeasureptemp}
M(A)=\int_{\mathbb{S}^{d-1}}\int_0^{\infty}\mathds{1}_A(ru)q(r^p,u)r^{-\alpha-1}\upd r \sigma(\upd u), \ \ \ A\in\mathcal{B}(\mathbb{R}^d),
\end{equation}
where $\sigma$ is a finite Borel measure on $\mathbb{S}^{d-1}$ and $q:(0,\infty)\times \mathbb{S}^{d-1}\rightarrow(0,\infty)$ is a Borel function such that, for all $u\in\mathbb{S}^{d-1}$, $q(\cdot,u)$ is \emph{completely monotone} (see \cite{Rosinski2007}) and is called \emph{tempering function}. 
%
%i.e., $q$ is infinitely differentiable in $r$ and 
%\begin{equation}
%\label{eq:completmonotone}
%(-1)^n\frac{\partial^n}{\partial r^n}q(r,u)\ge0,
%\end{equation}
%for $n\in\mathbb{N}$. Furthermore, $q$ satisfies
%\begin{equation}\label{eq:qintcond}
%\int_0^1r^{1-\alpha}q(r^p,u)\upd r<\infty,\ \int_1^{\infty}r^{-1-\alpha}q(r^p,u)\upd r<\infty,
%\end{equation}
%and
%\begin{equation}\label{qlimitcond}
%\lim_{r\to\infty}q(r,u)=0.
%\end{equation}
%$q$ is called \emph{tempering function}. 
 If
\begin{equation}
\label{eq:propertemp}
\lim_{r\downarrow0}q(r,u)=1,
\end{equation}
for each $u\in\mathbb{S}^{d-1}$, then $\mu$ is called a \emph{proper} $p$-\emph{tempered} $\alpha$ \emph{stable distribution}. We call the class of proper $p$-tempered $\alpha$-stable distributions $TS_{\alpha}^p$.
\end{defn}

%\begin{defn}\label{defn:propertemp}
%Let $\alpha\in(0,2)$ and $p>0$. Let $\mu\in TS_{\alpha}^p$ with L\'evy measure $M$ as in \eqref{eq:Levymeasureptemp}. If
%\begin{equation}
%\label{eq:propertemp}
%\lim_{r\downarrow0}q(r,u)=1,
%\end{equation}
%for each $u\in\mathbb{S}^{d-1}$, then $\mu$ is called a \emph{proper} $p$-\emph{tempered} $\alpha$ \emph{stable distribution}. We call the class of proper $p$-tempered $\alpha$-stable distributions $TS_{\alpha}^p$.
%\end{defn}

The tempering function can be represented as
\begin{equation}\label{qQrelation}
q(r^p,u)=\int_0^{\infty}\mathrm{e}^{-r^ps}Q(\upd s|u),
\end{equation}
where $\{Q(\cdot|u)\}_{u\in\mathbb{S}^{d-1}}$ is a measurable family of Borel measures. $Q(\cdot|u)$ is a probability measure if and only if \eqref{eq:propertemp} holds, i.e., $\mu$ is proper.

Proper $p$-tempered $\alpha$-stable distributions $\mu$ are infinitely divisible. Hence, they induce L\'evy processes $L=\{L_t\}_{t\in\mathbb{R}_+}$ s.t.~$L_1\sim \mu$. We call $L$ a $p$-\emph{tempered} $\alpha$-\emph{stable L\'evy process}. In this paper, we focus on CARMA processes driven by proper tempered stable L\'evy processes. These particular processes align with the initial motivation by modifying the tails of stable distributions to make them lighter to, e.g., obtain finite moments.

Next, we introduce some further measures on $\mathbb{R}^d$ which are related to the L\'evy measure $M$. They are needed later in the proofs and for the simulation scheme of the series representation. We define the Borel measure
\begin{equation}
\label{eq:Qmeasure}
Q(A):=\int_{\mathbb{S}^{d-1}}\int_0^{\infty}\mathds{1}_A(ru)Q(\upd s|u)\sigma(\upd u), \ \ \ A\in\mathcal{B}(\mathbb{R}^d).
\end{equation}
Then $Q({0})=0$. We also define the Borel measure
\begin{equation}
\label{eq:Rmeasure}
R(A):=\int_{\mathbb{R}^{d}}\mathds{1}_A\left(\frac{v}{||v||^{1+1/p}}\right)||v||^{\alpha/p}Q(\upd v), \ \ \ A\in\mathcal{B}(\mathbb{R}^d).
\end{equation}
Also, $R({0})=0$. The measure $R$ is called the \emph{Rosi\'nski measure}, see \cite{Rosinski2007}. We have the change of variable formula
\begin{equation}
\label{eq:changeofvariable}
\int_{\mathbb{R}^{d}}F(w)R(\upd w)=\int_{\mathbb{R}^{d}}F\left(\frac{v}{||v||^{1+1/p}}\right)||v||^{\alpha/p}Q(\upd v),
\end{equation} 
for any Borel function $F$ in the sense that if and only if one side exists then so does the other and both are equal. There is a one-to-one relationship between the L\'evy measure and the Rosi\'nski measure, see \cite[]{Grabchak2016book}.
For $p=1$ we have the special case of tempered $\alpha$-stable distributions as discussed in the seminal paper of \cite{Rosinski2007}. The $p$-tempered stable distributions we discuss constitute a specific subset of the generalized tempered stable distributions introduced by \cite{Rosinski2010} but have additional structure and cover various interesting subclasses. Moreover, many theoretical results that are available for tempered stable distributions of \cite{Rosinski2007} have an extended version in \cite{Grabchak2016book} while a general result for the generalized class of \cite{Rosinski2010} lacks. In the following subsection, we introduce some one-dimensional examples which can be found in the literature and are used in the Monte Carlo experiment in Section \ref{sec:MCstudy}.

We briefly review different method of series representations to sample paths of tempered stable distributions. The method we propose in Section \ref{sec:results} for general $p$ is an extension of \cite{Rosinski2007} for $p=1$. There are a few other well-known representations (inverse L\'evy method, rejection method, and thinning method) which are compared in \cite{Imai20114411,Kawai2011}. The inverse method and the rejection method (as a series representation, not to be confused with the rejection sampling of \cite{Grabchak2019}) have been extended by \cite{Rosinski2010} for general tempered stable distributions which include $TS_{\alpha}^p$. We opt for the aforementioned series representation because the inverse method is numerically more challenging, see \cite{Imai:2013:NIL:2489318.2489552}, and the rejection method may have a high number of rejections if the level of truncation is small, see Section \ref{sec:MCstudy} for more details. For further numerical properties of series representations see also \cite{Kawai2021,Yuan2021}.

%\subsection{Examples}\label{subsec:examples}

\begin{example}\label{expl:TSS}
The first special case is the $p$-tempered stable subordinator (pTSS). It is constructed from the stable subordinator which is a non-negative, increasing L\'evy process with $\alpha$-stable marginals. In this case, $\alpha$ needs to be in $(0,1)$. Its L\'evy measure is 
%\begin{equation}\label{eq:LevySS}
%L(\upd z)=\frac{\delta}{z^{1+\alpha}}\mathds{1}_{(0,\infty)}(z)\upd z,
%\end{equation}
%where $\delta>0$ is a scale parameter. For exponential tempering, the L\'evy measure of the pTSS distribution is given by
\begin{equation}\label{eq:LevyTSS}
M(\upd z)=\frac{\mathrm{e}^{-(\lambda z)^p}\delta}{z^{1+\alpha}}\mathds{1}_{(0,\infty)}(z)\upd z,
\end{equation}
where $\delta>0$ is a scale parameter and $\lambda>0$ a tempering parameter.
%The Rosi\'nski measure is given by
%\begin{equation}
%\label{eq:RosinskiTSS}
%R(\upd w) =\delta\lambda^{\alpha}\mathds{1}\left(\frac{1}{\lambda}\in\upd w\right).
%\end{equation}
%The corresponding measure $Q$ is given by
%\begin{equation}
%\label{eq:QmeasTSS}
%Q(\upd v)=\delta\mathds{1}(\lambda\in\upd v).
%\end{equation}
%The $m$-th order cumulants are given by
%\begin{equation}\label{eq:cumsTSS}
%\kappa_m=\Gamma(m-\alpha)\frac{\delta}{\lambda^{m-\alpha}},\ \ \ m\in\mathbb{N}.
%\end{equation}
%(Recall that $\kappa_1$ is the mean and $\kappa_2$ the variance.)
\end{example}

\begin{example}\label{expl:CTS}
One-dimensional classical $p$-tempered stable (pCTS) distributions are defined by their L\'evy measure 
\begin{equation}\label{eq:LevyCTS}
M(\upd z)=\left(\frac{\mathrm{e}^{-(\lambda_+z)^p}\delta_+}{z^{1+\alpha}}\mathds{1}_{(0,\infty)}(z)+\frac{\mathrm{e}^{-(\lambda_-|z|)^p}\delta_-}{|z|^{1+\alpha}}\mathds{1}_{(-\infty,0)}(z)\right)\upd z.
\end{equation}
$\alpha\in(0,2)$ is the stability parameter, $\delta_+,\delta_->0$ are scaling parameters, $\lambda_+,\lambda_->0$ are tempering parameters. The indices $+$ and $-$ refer to the positive and negative tails of the distribution.
%The Rosi\'nski measure is given by
%\begin{equation}
%\label{eq:RosinskiCTS}
%R(\upd w) =\delta_+\lambda_+^{\alpha}\mathds{1}\left(\frac{1}{\lambda_+}\in\upd w\right)+\delta_-\lambda_-^{\alpha}\mathds{1}\left(-\frac{1}{\lambda_-}\in\upd w\right).
%\end{equation}
%The corresponding measure $Q$ is given by
%\begin{equation}
%\label{eq:QmeasCTS}
%Q(\upd v)=\delta_+\mathds{1}(\lambda_+\in\upd v)+\delta_-\mathds{1}(-\lambda_-\in\upd v).
%\end{equation}
%The $m$-th order cumulants are given by
%\begin{equation}\label{eq:cumsCTS}
%\kappa_m=\Gamma(m-\alpha)\frac{\delta_+}{\lambda_+^{m-\alpha}}+(-1)^m\Gamma(m-\alpha)\frac{\delta_-}{\lambda_-^{m-\alpha}},
%\end{equation}
%for $m\ge 2$ and $\kappa_1=0$.
\end{example}

\begin{example}\label{expl:GTS}
Third, we discuss the gamma $p$-tempered stable stable distribution (p$\Gamma$TS), see \cite{Terdik2006} and \cite{Grabchak2016book}. For simplicity, we only consider the subordinator case which is one-sided with $\alpha\in(0,1)$. The L\'evy measure is
\begin{equation}
\label{eq:LevyGTS}
M(\upd z)=\left(\frac{z^p}{\lambda}+1\right)^{-\beta/p}z^{-1-\alpha}\mathds{1}_{(0,\infty)}(z)\upd z,
\end{equation}
where $\beta>0,\lambda>0$ are tempering parameters.
%The Rosi\'nski measure is given by
%\begin{equation}
%\label{eq:RosinskiGTS}
%R(\upd w) =\frac{p\lambda^{\beta/p}}{\Gamma(\beta/p)}\mathrm{e}^{-\lambda/w^p}w^{-1-\alpha-\beta}\mathds{1}_{(0,\infty)}(w)\upd w.
%\end{equation}
%The corresponding measure $Q$ is given by
%\begin{equation}
%\label{eq:QmeasGTS}
%Q(\upd v)=\frac{\lambda^{\beta/p}}{\Gamma(\beta/p)}\mathrm{e}^{-\lambda v}v^{-1+\beta/p}\mathds{1}_{(0,\infty)}(v)\upd v.
%\end{equation}
%Thus, $Q$ is absolutely continuous with respect to the Lebesgue measure where the density is equal to the density of the $\Gamma\left(\frac{\beta}{p},\lambda\right)$ distribution, which motivates the name ``gamma tempering''.
%The $m$-th order cumulants exist if $m<\alpha+\beta$ and are given by
%\begin{equation}\label{eq:cumsGTS}
%\kappa_m=\frac{\lambda^{\frac{m-\alpha}{p}}\Gamma\left(\frac{m-\alpha}{p}\right)\Gamma\left(\frac{\alpha+\beta-1}{p}\right)}{p\Gamma\left(\frac{\beta}{p}\right)},
%\end{equation}
%for all $m\ge1$.
\end{example}

\section{CARMA processes}\label{sec:carma}
In this section, we briefly collect facts for general CARMA processes from \cite{Brockwell2001,Brockwell2011,Brockwell2014}.

Let $\bm A\in\mathbb{R}^{\bar{p}\times {\bar{p}}}$, $\bm b\in\mathbb{R}^{\bar{p}}$, $\bm e_{\bar{p}} \in\mathbb{R}^{\bar{p}}$ be defined by
\begin{equation}
\label{eq:Amatrix}
\bm A:=\begin{pmatrix} 0&1&0&\cdots&0\\0&0&1&\cdots&0\\ \vdots&\vdots&\vdots&\ddots&\vdots\\ -a_{\bar{p}}&-a_{{\bar{p}}-1}&-a_{{\bar{p}}-2}&\cdots&-a_1\end{pmatrix},\ \ \ \bm a:=\begin{pmatrix}a_1\\b_2\\\vdots\\a_{{\bar{p}}-1}\\a_{{\bar{p}}}\end{pmatrix},\ \ \ \bm b:=\begin{pmatrix}b_0\\b_1\\\vdots\\b_{{\bar{p}}-2}\\b_{{\bar{p}}-1}\end{pmatrix},\ \ \ \bm e_{\bar{p}}:=\begin{pmatrix}0\\0\\\vdots\\0\\1\end{pmatrix},
\end{equation}
where $a_1,\ldots,a_{\bar{p}},b_0,\ldots, b_{{\bar{p}}-1}\in\mathbb{R}$ such that $b_{\bar{q}}=1$, ${\bar{q}}\le {\bar{p}}-1$ and $b_k=0$ for $k>{\bar{q}}$. Define the autoregressive and the moving average polynomials by
\begin{equation}
\label{eq:armapoly}
a(z):=z^{\bar{p}}+a_1z^{{\bar{p}}-1}+\cdots+a_{\bar{p}},\ \ \ b(z):=b_0+b_1z+\cdots+b_{\bar{q}}z^{\bar{q}},\ \ \ b_{\bar{q}}=1.
\end{equation}
Let $\lambda_1,\ldots,\lambda_{\bar{p}}$ be the eigenvalues of $\bm A$.
%It holds that
%\begin{equation}
%\label{eq:afactorized}
%a(z)=(z-\lambda_1)(z-\lambda_2)\cdots(z-\lambda_{\bar{p}}),
%\end{equation}
%i.e., the eigenvalues of $\bm A$ are also the roots of the AR polynomial $a(z)$.

We define a L\'evy-driven CARMA$({\bar{p}},{\bar{q}})$ process $\{Y_t\}_{t\in\mathbb{R}}$ in $\mathbb{R}$ by
\begin{equation}
\label{eq:carmaY}
Y_t:=\bm b^{\mathrm{T}}\bm X_t,
\end{equation}
for $t\in\mathbb{R}$, where $\{\bm X_t\}_{t\in\mathbb{R}}$ solves
\begin{equation}
\label{eq:carmaX}
\bm X_t=\mathrm{e}^{\bm A(t-s)}\bm X_s+\int_s^t\mathrm{e}^{\bm A(t-r)}\bm e_{\bar{p}}\upd L_r,\ \ \ s\le t,
\end{equation}
where $\{L_t\}_{t\in\mathbb{R}}$ is a one-dimensional two-sided L\'evy process defined on $\mathbb{R}$.

\begin{ass}\label{ass:carma}
The two polynomials $a(z)$ and $b(z)$ have no common roots and all roots of $a(z)$ are real, negative and distinct.
\end{ass}
Under this assumption the L\'evy-driven CARMA process $\{Y_t\}_{t\in\mathbb{R}}$ is unique and strictly stationary and can be written as 
\begin{equation}
\label{eq:carmaYg}
Y_t=\int_{-\infty}^{\infty}g(t-s)\upd L_s,
\end{equation}
where the function 
\begin{equation}
\label{eq:kernelg}
g(t)=\bm b^{\mathrm{T}}\mathrm{e}^{\bm A t}\bm e_{\bar{p}} \mathds{1}_{[0,\infty)}(t)
\end{equation}
is known as the \emph{kernel} of $\{Y_t\}$.
Moreover, by \cite{Brockwell2011}, we have
\begin{lem}\label{lem:OUdecomp}
Under Assumption \ref{ass:carma}, the CARMA$({\bar{p}},{\bar{q}})$ process can be written as the sum of $\bar{p}$ dependent and possibly complex-valued $CAR(1)$, i.e., Ornstein-Uhlenbeck (OU) processes. That is,
\begin{equation}
\label{eq:OUdecomp}
Y_t=\sum_{k=1}^{\bar{p}}Y_t^{(k)}:=\sum_{k=1}^{\bar{p}}\int_{-\infty}^t\alpha_k\mathrm{e}^{\lambda_k(t-s)}\upd L_s,
\end{equation}
where
\begin{equation}
\label{eq:alphak}
\alpha_k=\frac{b(\lambda_k)}{a'(\lambda_k)}, \ \ \ k=1,\ldots,{\bar{p}},
\end{equation}
and $a'$ denotes the first derivative of $a$.
\end{lem}

\begin{defn}\label{defn:CARMATS}
Let $p>0$, $\alpha\in(0,2)$. Let $\{L_t\}_{t\in\mathbb{R}}$ be a $TS_{\alpha}^{p}$ process. Let ${\bar{q}}\le {\bar{p}}-1$ and let $a(z)$ and $b(z)$ be polynomials such that Assumption \ref{ass:carma} is satisfied. We call the process $\{Y_t\}_{t\in\mathbb{R}}$ defined by \eqref{eq:carmaY} and \eqref{eq:carmaX} \emph{CARMA}$({\bar{p}},{\bar{q}})$\emph{-pTS process} (or short \emph{CARMA-pTS process}). For $\bar{p}=1,\bar{q}=0$, $CAR(1)$ processes are called \emph{OU-pTS processes}.
\end{defn}

%We note that
%\begin{equation}
%\label{eq:expYt}
%\mathbb{E}[Y_t]=\frac{b_{\bar{p}-1}}{a_{\bar{p}}}\mathbb{E}[L_1]
%\end{equation}
%and
%\begin{equation}
%\label{eq:covYt}
%\mathbb{C}ov[Y_t,Y_{t+h}]=\bm b^{\mathrm{T}}\exp(\bm A|h|)\bm \Sigma \bm b\mathbb{V}ar[L_1],
%\end{equation}
%where 
%\begin{equation}
%\label{eq:CARMASigma}
%\bm \Sigma=\int_0^{\infty}\exp(\bm Ay)\bm e\bm e^{\mathrm{T}}\exp(\bm A^{\mathrm{T}}y)\upd y.
%\end{equation}

We end the section with a notational remark highlighting the difference between $p$ (the tempering parameter) and $\bar{p}$ (the AR lag order) and the difference between $\alpha$ (the stability index) and the $\alpha_k$'s from \eqref{eq:alphak}.

\section{Theoretical results}\label{sec:results}
This section is divided into two subsections. In the first, we prove the main theorem for a series representation for $p$-tempered stable L\'evy distributions. In the second, we discuss resulting algorithms and approximation errors for simulating CARMA-pTS processes.
\subsection{Series representation for $p$-tempered stable distributions}\label{subsec:ptempseries}
In this subsection we derive a series representation for $p$-tempered $\alpha$-stable distributions and L\'evy processes. We generalize the results of \cite{Rosinski2007} for tempered stable distributions, i.e., $p=1$ and \cite{Bianchi2011} for $p=2$.% We prove the theorem in a general form for the whole L\'evy process $L$. 
\begin{thm}\label{thm:series}
Let $p>0$ and $T>0$. Let $\{E_j\}$ and $\{E_j'\}$ be i.i.d.~sequences of standard exponentials. Let $\{T_j\}$ be an i.i.d.~sequence of uniforms on $(0,T)$. Let $\{U_j\}$ be an i.i.d.~sequence of uniforms on $(0,1)$. Let $\{V_j\}$ be an i.i.d.~sequence of random vectors with distribution $Q/||\sigma||$. We assume all sequences to be independent. Set $\Gamma_j:=E_1'+\cdots+E_j'$. $\{\Gamma_j\}$ forms a Poisson point process on $(0,\infty)$ with the Lebesgue intensity measure.
\begin{enumerate}
	\item[(i)] If $\alpha\in(0,1)$, or if $\alpha\in[1,2)$ and $Q$ is symmetric, define
	\begin{equation}
	\label{eq:series1}
	L_t:=\sum_{j=1}^{\infty}\mathds{1}_{(0,t]}(T_j)\left(\left(\frac{\alpha \Gamma_j}{||\sigma||T}\right)^{-1/\alpha}\wedge\frac{E_j^{1/p}U_j^{1/\alpha}}{||V_j||^{1/p}}\right)\frac{V_j}{||V_j||}, \ \ \ t\in[0,T].
	\end{equation}
	Then the series \eqref{eq:series1} converges a.s.~uniformly in $t\in[0,T]$ to a L\'evy process $L$ such that $L_1\sim\mu\in TS_{\alpha}^p$ with L\'evy measure $M$ as in \eqref{eq:Levymeasureptemp}. 
	\item[(ii)] If $\alpha\in[1,2)$ and if $Q$ is non-symmetric and $1+p\neq\alpha$ and additionally assuming $\int_{\mathbb{R}^d}||w||R(\upd w)<\infty$ for $\alpha\in(1,2)$ and $\int_{\mathbb{R}^d}||w||\left|\log||w||\right|R(\upd w)<\infty$ for $\alpha=1$, define
		\begin{equation}
	\label{eq:series2}
	L_t:=\sum_{j=1}^{\infty}\mathds{1}_{(0,t]}(T_j)\left(\left(\frac{\alpha \Gamma_j}{||\sigma||T}\right)^{-1/\alpha}\wedge\frac{E_j^{1/p}U_j^{1/\alpha}}{||V_j||^{1/p}}\right)\frac{V_j}{||V_j||}-\frac{t}{T}\left(\frac{\alpha j}{||\sigma||T}\right)^{-1/\alpha}x_0+tb_T, \ \ \ t\in[0,T],
	\end{equation}
	where 
	\begin{equation}
	\label{eq:x0}
	x_0=\mathbb{E}\left[\frac{V_j}{||V_j||}\right]=||\sigma||^{-1}\int_{\mathbb{R}^d}u\sigma(\upd u),
	\end{equation}
	and
	\begin{equation}
	\label{eq:bT}
	b_T=\begin{cases}\alpha^{-1/\alpha}\zeta\left(\frac{1}{\alpha}\right)\left(||\sigma||T\right)^{1\alpha}T^{-1}x_0+\frac{1}{\alpha-1}\Gamma\left(\frac{1+p-\alpha}{p}\right)x_1,& 1<\alpha<2,\\
	\left(\frac{\gamma+p}{p}+\log(||\sigma||T)\right)x_1-\int_{\mathbb{R}^d}w\log||w||R(\upd w),&\alpha=1,
	\end{cases}
	\end{equation}
	where
	\begin{equation}
	\label{eq:x1}
	x_1=\int_{\mathbb{R}^d}wR(\upd w),
	\end{equation}
	and $\zeta$ denotes the Riemann zeta function. Then the series \eqref{eq:series2} converges a.s.~uniformly in $t\in[0,T]$ to a L\'evy process $L$ such that $L_1\sim\mu\in TS_{\alpha}^p$ with L\'evy measure $M$ as in \eqref{eq:Levymeasureptemp}. 
\end{enumerate}
\end{thm}

For simulation, the infinite series representations \eqref{eq:series1} and \eqref{eq:series2} need to be truncated. We follow \cite{Imai20114411} to derive the truncation error for $p$-tempered $\alpha$-stable L\'evy processes.

\begin{thm}\label{thm:trunc}
Let $p>0$, $\alpha\in(0,2)$ and $T>0$. Let $L$ be a $p$-tempered $\alpha$-stable L\'evy process with L\'evy measure $M$. For $n\in\mathbb{N}$, let 
\begin{equation}
\label{eq:truncseries}
L_t^{(n)}:=\sum_{\Gamma_j\le Tn}\mathds{1}_{(0,t]}(T_j)H(\Gamma_j;(V_j,E_j,U_j))-\frac{t}{T}\mathbb{E}\left[H(\Gamma_j;(V_j,E_j,U_j))\right],\ \ \ t\in[0,T],
\end{equation}
where $T_j,\Gamma_j,V_j,E_j,U_j$ are as in Theorem \ref{thm:series} and $H$ given in \eqref{eq:Hfunction}. Then, $L^{(n)}$ is a compound Poisson process with characteristic triplet $(0,0,M_{n})$ and converges to the characteristic triplet $(0,0,M)$ as $n\to\infty$, where 
\begin{equation}
\label{eq:truncmeasure}
M_{n}((x,\infty)B)=\left(n\frac{\alpha}{||\sigma||}x^{\alpha}\wedge1\right)M((x,\infty)B),
\end{equation}
for $x>0$ and $B\in\mathcal{B}(\mathbb{S}^{d-1})$.
\end{thm}

\subsection{CARMA-pTS process}\label{subsec:ptempCARMA}

In this subsection, we discuss path simulation of stationary CARMA-pTS processes, i.e., CARMA processes driven by a $p$-tempered $\alpha$-stable L\'evy process. We follow \cite{Kawai2017} to derive a series representation with the help of Theorem \ref{thm:series}. Let $\{Y_t\}_{t\in\mathbb{R}}$ be a CARMA$({\bar{p}},{\bar{q}})$ process driven by $\{L_t\}_{t\in\mathbb{R}}$, a one-dimensional $p$-tempered $\alpha$-stable L\'evy process on the real line. Let $T>0$. We aim to simulate the path on $[0,T]$. The method of \cite{Kawai2017} is based on the following decomposition
\begin{equation}
\label{eq:CARMAdecomp}
Y_t=Y_t(\kappa,n)+Q_t(n)+R_t(\kappa,n),
\end{equation} 
where $\kappa>0$ and $n\in\mathbb{N}$. $n$ denotes the level of truncation of small jumps as in Theorem \ref{thm:series}. $\kappa$ is the threshold of jump times. $Y_t(\kappa,n)$ is exactly simulatable with a series representation. $Q_t(n)$ are the very small jumps which may be approximated by a Gaussian CARMA process. $R_t(\kappa,n)$ are the jumps which occurred  before time $-\kappa$ and are to be discarded (or replaced by their mean). There are two different methods for simulation of CARMA-pTS processes to use depending on whether the tempered stable L\'evy process is a subordinator or not. We start with the case when $L$ is not a subordinator.

Let $M$ denote the L\'evy measure for a one-dimensional $TS_{\alpha}^p$ process as in \eqref{eq:Levymeasureptemp}. Note that $M(\mathbb{R})=\infty$ and the variance $\int_{\mathbb{R}}z^2M(\upd z)<\infty$. Let $\{M_n\}_{n\in\mathbb{N}}$ be a family of measures approximating $M$ such that for each $n$, $M_n(\mathbb{R})=n$ and $M_n(A)\le M(A)$ for $A\in\mathcal{B}(\mathbb{R})$. An example of such a family is the family $\{M_{n}\}$ of Theorem \ref{thm:trunc}. Define the variance of discarded jumps 
\begin{equation}
\label{eq:vardiscjumps}
\sigma_n^2:=\int_{\mathbb{R}}z^2(M-M_n)(\upd z).
\end{equation}
%Furthermore, let $N(\upd z, \upd t)$ and $N_n(\upd z, \upd t)$ be Poisson random counting measures with intensities $M(\upd z)\upd t$ and $M_n(\upd z)\upd t$, respectively, and $\widetilde{N}(\upd z, \upd t)=N(\upd z, \upd t)-M(\upd z)\upd t$ and $\widetilde{N}_n(\upd z, \upd t)=N_n(\upd z, \upd t)-M_n(\upd z)\upd t$ the compensated Poisson random measures. The components in \eqref{eq:CARMAdecomp} are defined by
%\begin{align}
%\{Y_t(\kappa,n)\}_{t\in[0,T]}&:=\left\{\int_{-\kappa}^T\int_{\mathbb{R}}g(t-s)z\widetilde{N}_n(\upd z, \upd t)\right\}_{t\in[0,T]}\\
%\{Q_t(n)\}_{t\in[0,T]}&:=\left\{\int_{-\infty}^T\int_{\mathbb{R}}g(t-s)z(\widetilde{N}-\widetilde{N}_n)(\upd z, \upd t)\right\}_{t\in[0,T]}\\
%\{R_t(\kappa,n)\}_{t\in[0,T]}&:=\left\{\int_{-\infty}^{-\kappa}\int_{\mathbb{R}}g(t-s)z\widetilde{N}_n(\upd z, \upd t)\right\}_{t\in[0,T]}\\
%\end{align}

\cite{Kawai2017} showed that, if for each $c>0$
\begin{equation}
\label{eq:Asmussen}
\frac{1}{\sigma_n^2}\int_{\{z^2>c\sigma_n^2\}}z^2(M-M_n)(\upd z)\to0
\end{equation}
for $n\to\infty$, then
\begin{equation}
\label{eq:Wienerapprox}
\left\{\frac{Q_t(n)}{\sigma_n}\right\}_{t\in[0,T]}\stackrel{\mathcal{D}}{\rightarrow}\left\{\int_{-\infty}^T\int_{\mathbb{R}}g(t-s)\upd W_s\right\}_{t\in[0,T]},
\end{equation}
where $\{W_t\}_{t\in[0,T]}$ is a two-sided Brownian motion. Furthermore, it holds that for each $n\in\mathbb{N}$, $\{R_t(\kappa,n)\}_{t\in[0,T]}$ converges in probability to the zero process uniformly on $[0,T]$ as $\kappa\to\infty$. Section 3.3 of \citet{cohen2007} shows that \eqref{eq:Asmussen} holds (for multivariate) tempered $\alpha$-stable processes. Their result can be easily extended to $p$-tempered stable processes.

Therefore, together with Theorems \ref{thm:series} and \ref{thm:trunc} we can propose the following approximation schemes. We stress the difference between the variance of discarded jumps $\sigma_n^2$ and $||\sigma||$ of the Borel measure $\sigma$.
\begin{scheme}\label{cor:CARMA}
Let $T>0,\kappa>0$. Let $p>0$, $\alpha\in(0,2)$. Let $\{L_t\}_{t\in\mathbb{R}}$ be a $TS_{\alpha}^{\bar{p}}$ process with L\'evy measure $M$. Let $\{W_t\}_{t\in\mathbb{R}}$ be a Brownian motion. Let ${\bar{q}}\le {\bar{p}}-1$ and let $a(z)$ and $b(z)$ be polynomials such that Assumption \ref{ass:carma} is satisfied and $\alpha_k$ and $\lambda_k$ for $k=1,\ldots,\bar{p}$ as in Lemma \ref{lem:OUdecomp}. Let $\{Y_t\}_{t\in\mathbb{R}}$ be a CARMA-pTS process.

Let $n\in\mathbb{N}$ and $\sigma_n^2$ as in \eqref{eq:vardiscjumps} and $\sigma_n:=\sqrt{\sigma_n^2}$.
Let $Z_{\kappa,n}$ be a Poisson random variable with mean $(T+\kappa)n$. Let $\{\Lambda_{(j)}\}_{j=1,\ldots,Z_{\kappa,n}}$ be the ascending order statistic of $Z_{\kappa,n}$ i.i.d.~uniform random variables on $(0,\alpha n/||\sigma||)$. Let $\{E_j\}$ be a i.i.d.~sequence of standard exponentials. Let $\{T_j\}$ be an i.i.d.~sequence of uniforms on $(-\kappa,T)$. Let $\{U_j\}$ be an i.i.d.~sequence of uniforms on $(0,1)$. Let $\{V_j\}$ be an i.i.d.~sequence of random vectors with distribution $Q/||\sigma||$. We assume all sequences to be independent.

Then, $Y_t$ for $t\in[0,T]$ can be approximated by
\begin{enumerate}
\item[(i)] if $\alpha\in(0,1)$, or if $\alpha\in[1,2)$ and $Q$ is symmetric,
\begin{equation}\label{eq:CARMA1}
\sum_{k=1}^{\bar{p}}\sum_{j=1}^{Z_{\kappa,n}}\alpha_k\mathrm{e}^{\lambda_k(t-T_j)}\mathds{1}_{(-\kappa,t]}(T_j)\left(\Lambda_{(j)}^{-1/\alpha}\wedge\frac{E_j^{1/p}U_j^{1/\alpha}}{||V_j||^{1/p}}\right)\frac{V_j}{||V_j||}+\sigma_n\int_{-\infty}^t\alpha_k\mathrm{e}^{\lambda_k(t-s)}\upd W_s
\end{equation}
\item[(ii)] if $\alpha\in[1,2)$ and if $Q$ is non-symmetric and $1+p\neq\alpha$, additionally assuming $\int_{\mathbb{R}^d}||w||R(\upd w)<\infty$ for $\alpha\in(1,2)$ and $\int_{\mathbb{R}^d}||w||\left|\log||w||\right|R(\upd w)<\infty$ for $\alpha=1$,
\begin{align}
\sum_{k=1}^{\bar{p}}&\left[\sum_{j=1}^{Z_{\kappa,n}}\alpha_k\mathrm{e}^{\lambda_k(t-T_j)}\mathds{1}_{(-\kappa,t]}(T_j)\left(\Lambda_{(j)}^{-1/\alpha}\wedge\frac{E_j^{1/p}U_j^{1/\alpha}}{||V_j||^{1/p}}\right)\frac{V_j}{||V_j||}+\alpha_k\frac{1-\mathrm{e}^{\lambda_k(t+\kappa)}}{\lambda_k(T+\kappa)}\left(\frac{\alpha j}{||\sigma||(T+\kappa)}\right)^{-1/\alpha}x_0\right.\\
& - \left.\alpha_k \frac{1-\mathrm{e}^{\lambda_k(t+\kappa)}}{\lambda_k}b_{T+\kappa}+\sigma_n\int_{-\infty}^t\alpha_k\mathrm{e}^{\lambda_k(t-s)}\upd W_s\right]\label{eq:CARMA2}
\end{align}
where $b_{T+\kappa}$ as in \ref{eq:bT}, and with $x_0,x_1$ are as in Theorem \ref{thm:series}.
\end{enumerate}
\end{scheme}

For a tempered stable subordinator, i.e., a non-negative BDLP, together with a non-negative kernel function \eqref{eq:kernelg} the resulting stationary CARMA process is non-negative. We could use the same approximation scheme as in \eqref{eq:CARMA1}. However, as \cite{Kawai2017} noted, the Gaussian CARMA process would destroy its non-negativity. Therefore, we simulate $Y_t(\kappa,n)+\mathbb{E}[Q_t(n)]+\mathbb{E}[R_t(\kappa,n)]$ instead of discarding $Q_t(n)$ and $R_t(\kappa,n)$.

\begin{scheme}\label{cor:CARMAsub}
Let $T>0,\kappa>0$. Let $p>0,\alpha\in(0,1)$. Let $\{L_t\}_{t\in\mathbb{R}}$ be a $TS_{\alpha}^{p}$ subordinator with L\'evy measure $M$. Let ${\bar{q}}\le {\bar{p}}-1$ and let $a(z)$ and $b(z)$ be polynomials such that Assumption \ref{ass:carma} is satisfied and $\alpha_k$ and $\lambda_k$ for $k=1,\ldots,\bar{p}$ as in Lemma \ref{lem:OUdecomp}. Let $\{Y_t\}_{t\in\mathbb{R}}$ be a CARMA-pTS process.
Let $n\in\mathbb{N}$. Let the random sequences be as in Scheme \ref{cor:CARMA}.
Then, $Y_t$ for $t\in[0,T]$ can be approximated by
\begin{equation}\label{eq:CARMAsub}
\sum_{k=1}^{\bar{p}}\sum_{j=1}^{Z_{\kappa,n}}\alpha_k\mathrm{e}^{\lambda_k(t-T_j)}\mathds{1}_{(-\kappa,t]}(T_j)\left(\Lambda_{(j)}^{-1/\alpha}\wedge\frac{E_j^{1/p}U_j^{1/\alpha}}{||V_j||^{1/p}}\right)-\int_{\mathbb{R}_+}z(M-M_n)(\upd z)\frac{\alpha_k}{\lambda_k}-\int_{\mathbb{R}_+}zM_n(\upd z)\frac{\alpha_k}{\lambda_k}\mathrm{e}^{\lambda_k(t+\kappa)}.
\end{equation}
\end{scheme}

\begin{rem}
In Scheme \ref{cor:CARMAsub}, we can omit the term $\frac{V_j}{||V_j||}$ because for a subordinator the L\'evy measure is only concentrated on $\mathbb{R}_+$ and thus numerator and denominator coincide.
\end{rem}

\begin{rem}
As mentioned above, $TS_{\alpha}^{p}$ distributions are part of generalized tempered stable distributions of \cite{Rosinski2010}. In case of a generalized tempered stable distribution that is not in  $TS_{\alpha}^{p}$ of course Theorem \ref{thm:series} does not apply anymore. However, \cite{Rosinski2010} derived series representations with the inverse L\'evy measure method and the rejection method. Then, \eqref{eq:CARMA1}, \eqref{eq:CARMA2} and \eqref{eq:CARMAsub} are modified to contain the corresponding terms of the series of Theorems 5.1 and 5.5 in \cite{Rosinski2010} (it is however necessary to check whether the generalized tempered stable distribution in question still fulfills the Wiener approximation assumption).
\end{rem}

\begin{rem}
The proposed Scheme \ref{cor:CARMAsub} can be extended to the case of multivariate CARMA processes (MCARMA) introduced in \cite{Marquardt2007}. In order to do so, we can replace Lemma \ref{ass:carma} with its multivariate version in Proposition 5.1 of \cite{Schlemm2012}. Using this, a stationary MCARMA process can be expressed as a sum of dependent, complex-valued, and multivariate OU processes. The resulting formulas are lengthy and omitted but follow the same principle as in Scheme \ref{cor:CARMAsub} above.
\end{rem}

The above allows us to easily derive a bound for the approximation mean-squared error, which obviously depends on the specific underlying background-driving tempered stable process.
\begin{cor}\label{cor:error}
Under the assumptions of Scheme \ref{cor:CARMA}, let $\widetilde{Y}_t(\kappa,n)$ denotes the approximation of $Y_t$ according to \eqref{eq:CARMA1} or \eqref{eq:CARMA2}. Then
\begin{align}
\label{eq:error}
\mathbb{E}\left[(Y_t-\widetilde{Y}_t(\kappa,n))^2\right]\le& -\sigma_n^2\left(\sum_{k=1}^{\bar{p}}\frac{1}{2\lambda_k}\right)\left(\sum_{k=1}^{\bar{p}}\alpha_k^2\right)+\left(\int_{\mathbb{R}}z(M-M_n)(\upd z)\right)^2\left(\sum_{k=1}^{\bar{p}}\frac{\alpha_k}{\lambda_k}\right)^2\\
&-\int_{\mathbb{R}}z^2M_n(\upd z)\left(\sum_{k=1}^{\bar{p}}\frac{\mathrm{e}^{2\lambda_k(\kappa+t)}}{2\lambda_k}\right)\left(\sum_{k=1}^{\bar{p}}\alpha_k^2\right) + \left(\int_{\mathbb{R}}zM_n(\upd z)\right)^2\left(\sum_{k=1}^{\bar{p}}\frac{\alpha_k\mathrm{e}^{\lambda_k(\kappa+t)}}{\lambda_k}\right)^2.%\\
%&+\sigma_n^2\bm b^{\mathrm{T}}\bm \Sigma \bm b,
\end{align}
%where $\bm\Sigma$ is given in \eqref{eq:CARMASigma}.
If $L$ is a subordinator as in Scheme \ref{cor:CARMAsub} and $\widetilde{Y}_t(\kappa,n)$ denote the approximation of $Y_t$ according to \eqref{eq:CARMAsub}, then 
\begin{equation}
\label{eq:suberror}
\mathbb{E}\left[(Y_t-\widetilde{Y}_t(\kappa,n))^2\right]\le-\sigma_n^2\left(\sum_{k=1}^{\bar{p}}\frac{1}{2\lambda_k}\right)\left(\sum_{k=1}^{\bar{p}}\alpha_k^2\right)-\int_{\mathbb{R}_+}z^2M_n(\upd z)\left(\sum_{k=1}^{\bar{p}}\frac{\mathrm{e}^{2\lambda_k(\kappa+t)}}{2\lambda_k}\right)\left(\sum_{k=1}^{\bar{p}}\alpha_k^2\right).
\end{equation}
\end{cor}

The formulas \eqref{eq:error} and \eqref{eq:suberror} appear unhandy but we can explicitly compute the error bounds for a specific parameter constellation with the help of the expressions for the first and second moments of the L\'evy measure to be found in Appendix \ref{app:moments}.

\section{Monte Carlo study}\label{sec:MCstudy}

In this section, we conduct several Monte Carlo experiments to test the presented sampling routines. We follow up on the illustrative examples introduced earlier in Section \ref{sec:distributions}. Unfortunately, there is no direct way to test whether a sample path was generated by a CARMA-pTS process. This is because the marginal distribution of a CARMA-pTS is not explicitly given. We therefore compare the empirical means and variances with their theoretical counterparts. Additionally, we inspect the empirical distribution of marginals given by the truncated series representation of a tempered stable L\'evy process.

%We start by comparing empirical with theoretical moments. For the empirical moments we simulate 10000 CARMA-pTSS paths $\{Y_t\}_{t\in[0,T]}$ with $T=100$ and compute mean and variance of $Y_t$ with $t=1$ and $t=100$.

\begin{example}[continues=expl:TSS]
We start with the pTSS as the BDLP. We consider a range of parameter constellations to gauge the influence of different parameter values on the accuracy of our method. We choose $\alpha\in\{0.5,0.8\}$, $p\in\{0.5,1,2\}$, $\delta=1$ and $\lambda=1$ as parameters for the pTSS. (In unreported results we see that the impact for scaling and tempering parameters is negligible.) To obtain a non-negative kernel function which fulfills Assumption \ref{ass:carma} we set $\bm a = (3,2)^{\mathrm{T}}$ and $\bm b=(3,1)^{\mathrm{T}}$. We simulate paths on $[0,T]$ with $T=100$. We make use of the decomposition $Y_t(\kappa,n)+\mathbb{E}[Q_t(n)]+\mathbb{E}[R_t(\kappa,n)]$ where $Y_t(\kappa,n)$ is the truncated series in \eqref{eq:CARMAsub} and $\mathbb{E}[Q_t(n)]$ and $\mathbb{E}[R_t(\kappa,n)]$ are the two remaining terms in \eqref{eq:CARMAsub}. The expectations are analytically available. However, the formulas are lengthy and relegated to Appendix \ref{app:moments}. For $Y_t(\kappa,n)$, next to standard random sequences $\{E_j\},\{U_j\},\{T_j\},\{\Lambda_j\}$, we need to simulate $V_j\sim Q(\upd v)/||\sigma||$. In the present case, $||\sigma||=\delta=1$ and thus $V_j=\lambda$ with probability 1. Therefore, $Y_t(\kappa,n)$ in \eqref{eq:CARMAsub} reduces to
\begin{equation}
\sum_{k=1}^{2}\sum_{j=1}^{Z_{\kappa,n}}\alpha_k\mathrm{e}^{\lambda_k(t-T_j)}\mathds{1}_{(-\kappa,t]}(T_j)\left(\Lambda_{(j)}^{-1/\alpha}\wedge\frac{E_j^{1/p}U_j^{1/\alpha}}{\lambda^{1/p}}\right),
\end{equation}
with $\lambda=1$. For the above choice of $\bm a$ and $\bm b$ we have that $\alpha_1=2,\alpha_2=-1, \lambda_1=-1,\lambda_2=-2$.
We compare the accuracy for levels of truncation $n\in\{10,100,1000,10000\}$. We set the threshold of jump times at $\kappa=100$.

Figure \ref{fig:CARMATSS} plots a typical paths of CARMA-pTSS process for $\alpha=0.5,n=10000,\kappa=100$. The top panel shows a path for $p=0.5$, the middle for $p=1$ and the bottom for $p=2$. The remaining parameters remain fixed at the aforementioned values. To ensure a fair comparison, we employ the same random sequence seed across all plots. We observe that for lower values of $p$, the path frequently traverses through smaller values, yet features higher peaks when compared to the paths corresponding to larger values of $p$. This observation aligns with an increase in the average value across the path as $p$ increases.

\begin{figure*}
	\centering
  \subfloat[][]{\includegraphics[width=.65\textwidth]{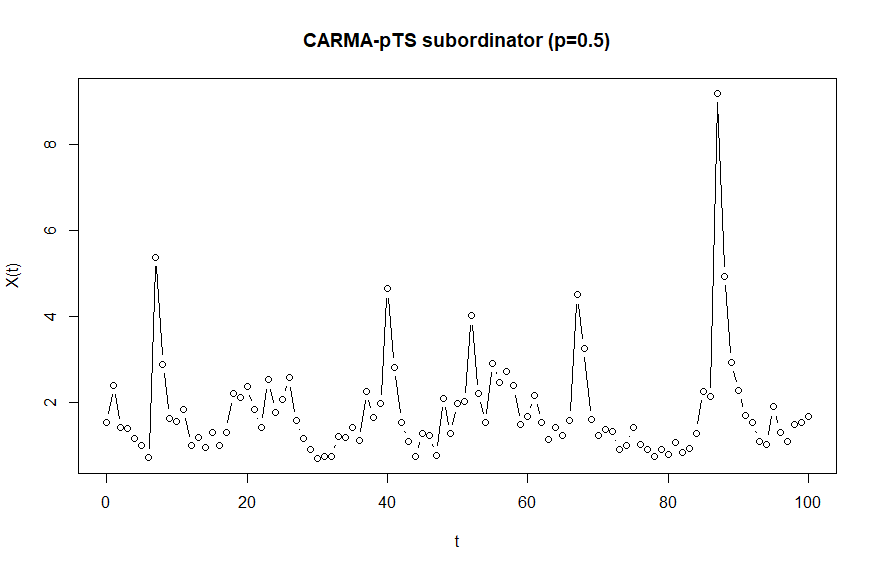}} \par
  
  \subfloat[][]{\includegraphics[width=.65\textwidth]{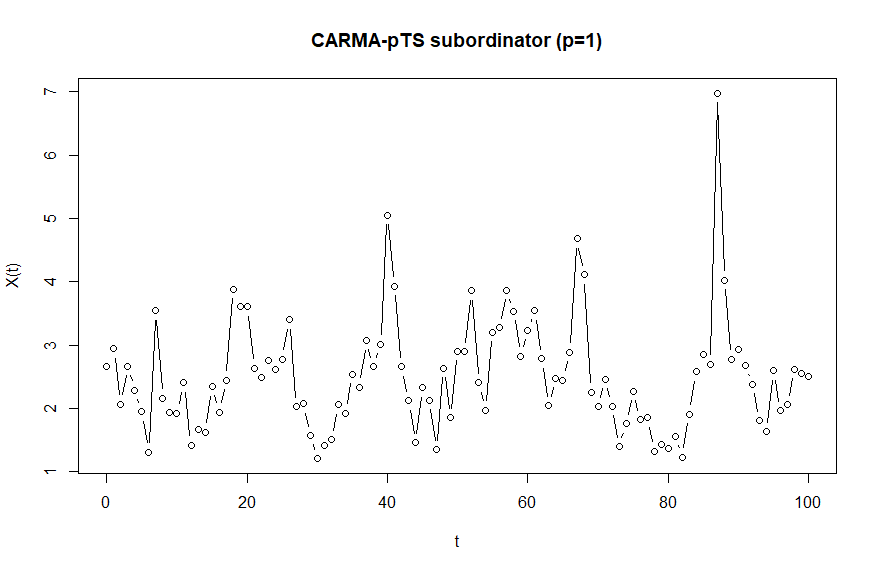}}\par

	\subfloat[][]{\includegraphics[width=.65\textwidth]{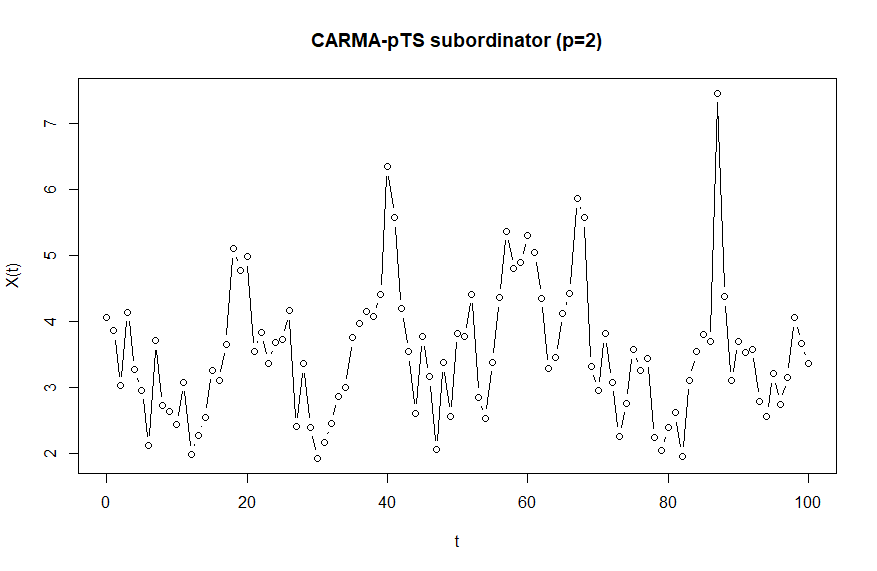}}\par
  
  \caption{Sample paths of the CARMA-pTSS process for different $p$ on $[0,100]$. Other parameters fixed at $\alpha=0.5,\delta=\lambda=1,\bm a = (3,2)^{\mathrm{T}},\bm b=(3,1)^{\mathrm{T}}$. We use \eqref{eq:CARMAsub} for simulation with $n=10000,\kappa=100.$}
  \label{fig:CARMATSS}
\end{figure*}

Table \ref{tab:pTSS} reports absolute accuracy of Monte Carlo estimations of $\mathbb{E}[Y_t]$ and $\mathbb{V}ar[Y_t]$, both evaluated at times $t=1$ and $t=100$. We compare the said constellations for $\alpha$, $p$ and $n$. For each parameter constellation, we conduct simulations involving 10000 CARMA-pTSS paths. The empirical mean and empirical variance are subsequently computed. The table's two center columns report the absolute accuracies defined as $\widehat{\mathbb{E}}[Y_t]-\mathbb{E}[Y_t]$, and the two right-most report $\widehat{\mathbb{V}ar}[Y_t]-\mathbb{V}ar[Y_t]$. We make the following observations. The estimations are accurate in most cases for both expectation and variance. Interestingly, the estimation results do not improve for increasing $n$. The simulation performs well even for relatively small values of $n$. This phenomenon might be attributed to the fact that, given $T=\kappa=100$, an average of 2000 jumps occurs per path for $n=10$. This number of jumps seemingly suffices for $Y_t$ to conform to its stationary distribution. We visualize this with Figure \ref{fig:densCARMATSS} which shows relative stability in the empirical density of $Y_{100}$ realizations. Unfortunately, we cannot compare the empirical distributions with the true distribution due to the lack of a closed-form function. We omit figures for the other constellations of $\alpha$ and $p$ because the pattern is the same.

\begin{table}[htbp]
  \centering
    \begin{tabular}{rrrrrrr}
		\toprule
    $\alpha$ & $p$& $n$ & \multicolumn{2}{c}{$\mathbb{E}$}       & \multicolumn{2}{c}{$\mathbb{V}ar$}   \\
          &       &       & $t=1$ & $t=100$ & $t=1$ & $t=100$ \\
					\midrule
    0.5   & 0.5   & 10    & 0.006 & 0.001 & -0.165 & 0.397 \\
          &       & 100   & 0.017 & -0.001 & 0.063 & 0.511 \\
          &       & 1000  & -0.026 & 0.014 & -0.013 & 0.078 \\
          &       & 10000 & 0.026 & 0.009 & 0.174 & 0.126 \\
          & 1     & 10    & 0.01 & -0.004 & 0.001 & -0.003 \\
          &       & 100   & -0.0003 & 0.002 & 0.013 & -0.007 \\
          &       & 1000  & -0.011 & 0.013 & -0.015 & 0.003 \\
          &       & 10000 & 0.011 & 0.005 & 0.014 & -0.011 \\
          & 2     & 10    & 0.008 & -0.001 & 0.007 & -0.001\\
          &       & 100   & -0.002 & 0.003 & 0.008 & -0.008 \\
          &       & 1000  & -0.008 & 0.013 & -0.013 & 0.006 \\
          &       & 10000 & 0.007 & 0.006 & 0.003 & -0.002 \\
    0.8   & 0.5   & 10    & 0.013 & -0.007 & -0.062 & 0.055 \\
          &       & 100   & 0.022 & 0.011 & 0.196 & 0.179 \\
          &       & 1000  & -0.031 & 0.005 & -0.144 & 0.231 \\
          &       & 10000 & 0.026 & 0.012 & 0.2 & 0.261 \\
          & 1     & 10    & 0.008 & -0.006 & -0.021 & -0.028 \\
          &       & 100   & 0.005 & 0.0006 & 0.024 & -0.017 \\
          &       & 1000  & -0.01 & 0.006 & -0.023 & 0.007 \\
          &       & 10000 & 0.011 & 0.003 & -0.002 & 0.0002 \\
          & 2     & 10    & 0.006 & -0.003 & -0.024 & -0.031 \\
          &       & 100   & 0.01 & -0.017 & 0.01 & -0.017 \\
          &       & 1000  & -0.015 & 0.0005 & -0.015 & 0.0005\\
          &       & 10000 & -0.005 & -0.003 & -0.005 & -0.003 \\
					\bottomrule
    \end{tabular}%
	  \caption{Accuracy estimations for expectations and variances based on 10000 sample paths of the CARMA-pTSS on $[0,100]$ for different $\alpha$ and $p$. Other parameters fixed at $\delta=\lambda=1,\bm a = (3,2)^{\mathrm{T}},\bm b=(3,1)^{\mathrm{T}}$. We use \eqref{eq:CARMAsub} for simulation for different $n$ with $\kappa=100.$}
		  \label{tab:pTSS}%
\end{table}%

\begin{figure*}
	\centering
\includegraphics[width=.6\textwidth]{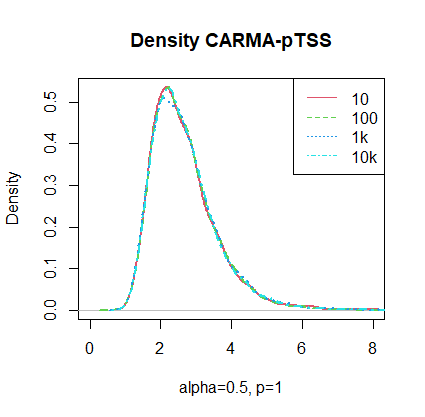}
  \caption{Empirical densities of 10000 realizations of $Y_{100}$ of a CARMA-pTSS for different $n$. Parameters fixed $\alpha=0.5,\delta=\lambda=1,\bm a = (3,2)^{\mathrm{T}},\bm b=(3,1)^{\mathrm{T}}$ and $\kappa=100$.}
  \label{fig:densCARMATSS}
\end{figure*}

The only exception with a worse accuracy is the case for $p=0.5$ for the estimation of the variance. Remarkably, this behavior remains consistent across various values of $n$ and persists for both examined $\alpha$ settings. This phenomenon can be attributed to the inherent nature of the problem: for small $p$ values, seldom yet substantial spikes occur in the paths. Consequently, the accuracy of variance estimation is compromised due to these outliers. To visualize this, Figure \ref{fig:boxplots} shows boxplots of the realizations of $Y_{100}$ for $\alpha=0.5$ and $n=10000$ for the different $p$s.   

\begin{figure*}
	\centering
  \subfloat[][]{\includegraphics[width=.33\textwidth]{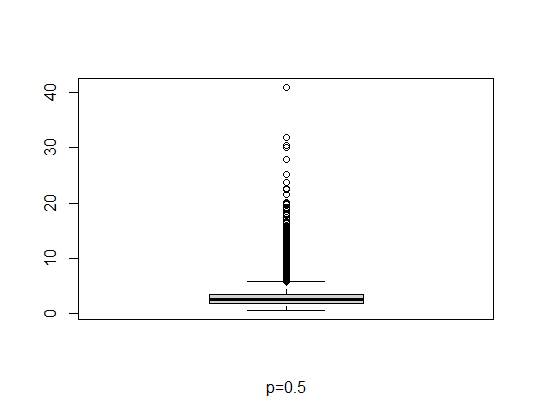}} \hfill
  \subfloat[][]{\includegraphics[width=.33\textwidth]{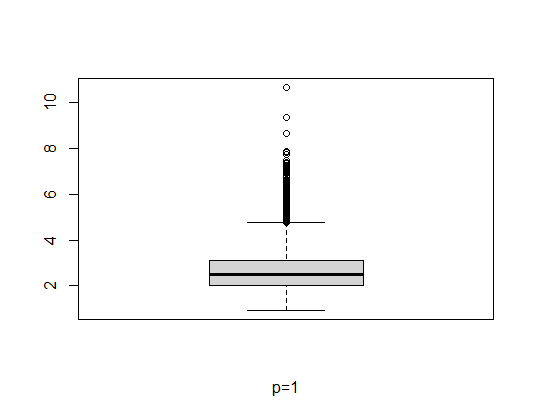}}\hfill
	\subfloat[][]{\includegraphics[width=.33\textwidth]{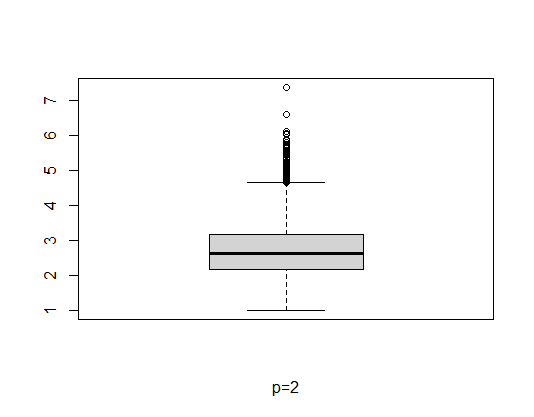}}\par
  \caption{Boxplots of 10000 realizations of $Y_{100}$ of a CARMA-pTSS for $p\in\{0.5,1,2\}$. Other parameters fixed $\alpha=0.5,\delta=\lambda=1,\bm a = (3,2)^{\mathrm{T}},\bm b=(3,1)^{\mathrm{T}}$. We use $n=10000,\kappa=100$ for simulation.}
  \label{fig:boxplots}
\end{figure*}

We extend the conducted simulation study to cases where $\kappa=1$ and $\kappa=10$, yielding qualitatively similar results, which are, however, omitted here for brevity. This demonstrates that for path simulations of a CARMA-pTSS process, modest values of $n$ and $\kappa$ turn out to be sufficient. Importantly, this is in contrast to the generation of i.i.d.~tempered stable random variates through series representations. To illustrate this, we simulate 1000 i.i.d.~random variates with a pTSS distribution for $p=1$. We subsequently compare the empirical distributions across varying $n$ values. Figure \ref{fig:densTSS} plots the empirical densities for $\alpha=0.5$ in panel (a) and $\alpha=0.8$ in panel (b). For reference, the black solid line shows the true density function. As above, set $\delta=\lambda=1$. As expected, the goodness-of-fit to the true distribution improves with increasing $n$. For $\alpha=0.5$ a value of $n=100$ seems to be sufficient. However, the situation changes considerably for $\alpha=0.8$ where even a large value of $n=100000$ fails to yield satisfactory goodness-of-fit results. 
For further comparison with the aforementioned Table \ref{tab:pTSS}, Table \ref{tab:TSS} reports the mean and variance accuracy for the i.i.d.~case. We observe that while the simulated variates exhibit a variance close to the true one for $\alpha=0.8$, there remains a substantial disparity in the mean, even for exceedingly large $n$ values.

We also touch upon different series representations such as the rejection method of \cite{Rosinski2010}. We omit a full discussion and presenting lengthy tables because the differences in accuracy is not striking. The general observation is that the series representation of Section \ref{sec:results} has a slightly smaller numerical error than the rejection method of \cite{Rosinski2010} which may also have the issue of a small number of non-rejections if $n$ is small or if a high number of jumps is needed. The observations are in line with those of \cite{Imai20114411} for $p=1$.

\begin{figure*}
	\centering
\subfloat[][]{\includegraphics[width=.5\textwidth]{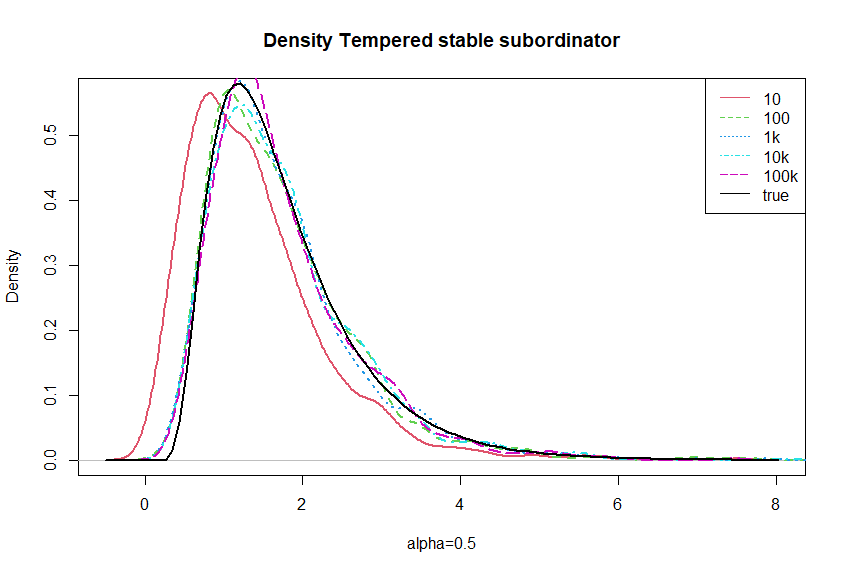}} \hfill
\subfloat[][]{\includegraphics[width=.5\textwidth]{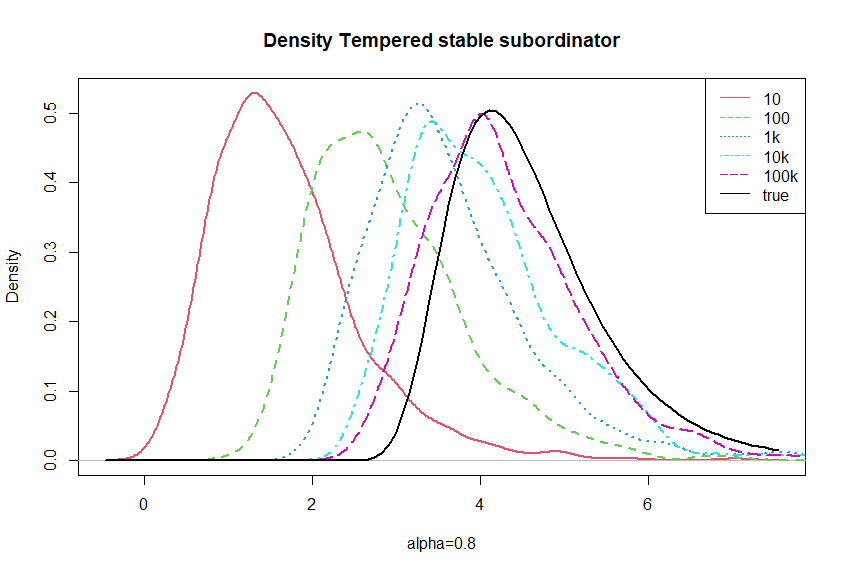}} \hfill
  \caption{Empirical densities of 1000 i.i.d.~realizations of a pTSS random variable for $\alpha\in\{0.5,0.8\}$. Other parameters fixed $p=1,\delta=\lambda=1,\bm a = (3,2)^{\mathrm{T}},\bm b=(3,1)^{\mathrm{T}}$. We compare different $n$ and set $\kappa=100$.}
  \label{fig:densTSS}
\end{figure*}

\begin{table}[htbp]
  \centering
    \begin{tabular}{rrrr}
		\toprule
    $\alpha$ & $n$ & $\mathbb{E}$       & $\mathbb{V}ar$   \\
					\midrule
    0.5       & 10    & -0.394 &  -0.095 \\
                & 100   & -0.038 & 0.073 \\
          & 1000  & -0.035 & 0.006 \\
                 & 10000 & -0.01 & -0.014  \\
					& 100000 & -0.03 &  -0.063 \\
    0.8      & 10    & -2.913 &  -0.164 \\
                 & 100   & -1.648 & 0.097 \\
                & 1000  & -0.977 &  -0.003\\
                 & 10000 & -0.533 &   0.043\\
								& 100000 & -0.312 &  -0.094\\
					\bottomrule
    \end{tabular}%
	  \caption{Accuracy estimations for expectations and variances based on 1000 i.i.d.~realizations of a pTSS random variable for $\alpha\in\{0.5,0.8\}$. Other parameters fixed $p=1,\delta=\lambda=1,\bm a = (3,2)^{\mathrm{T}},\bm b=(3,1)^{\mathrm{T}}$. We compare different $n$ and set $\kappa=100$.}
		  \label{tab:TSS}%
\end{table}%

\end{example}

\begin{example}[continues=expl:CTS]

We continue with the example of the pCTS distribution as the BDLP. As in Example \ref{expl:TSS}, we set $T=\kappa=100$, $\bm a = (3,2)^{\mathrm{T}}$ and $\bm b=(3,1)^{\mathrm{T}}$ and $p\in\{0.5,1,2\}$. In this instance, the parameters for the tempered stable distribution are $\delta_+=\delta_-=\lambda_+=\lambda_-=1$ and $\alpha\in\{1.4,1.8\}$. Given the symmetry of the measure $Q$ for these parameter choices, we leverage \eqref{eq:CARMA1} for simulation purposes. The sequence ${V_j}$ is sampled such that $\mathbb{P}[V_j=-\lambda_-]=\mathbb{P}[V_j=+\lambda_+]=0.5$. Figure \ref{fig:CARMACTS} presents a typical path for $\alpha=1.4$, $p=1$, and $n=10000$.

Table \ref{tab:pCTS} reports the estimated accuracy for both expectation and variance. The observations align with those made for CARMA-pTSS, with the additional insight that for $\alpha=1.8$, variance estimation proves to be noisy across all considered values of $p$. This contrasts with Table \ref{tab:pTSS} where this noisy behavior is only prevalent for $p=0.5$.

\begin{figure*}
	\centering
  \includegraphics[width=.75\textwidth]{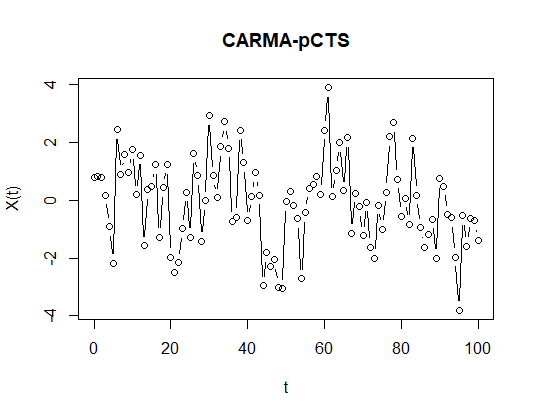}
  \caption{Sample paths of the CARMA-pCTS process on $[0,100]$. Parameters fixed at $\alpha=1.4,p=\delta_+=\delta_-=\lambda_+=\lambda_-=1,\bm a = (3,2)^{\mathrm{T}},\bm b=(3,1)^{\mathrm{T}}$. We use \eqref{eq:CARMA1} for simulation with $n=10000,\kappa=100.$}
  \label{fig:CARMACTS}
\end{figure*}

\begin{table}[htbp]
  \centering
  \begin{tabular}{rrrrrrr}
		\toprule
    $\alpha$ & $p$& $n$ & \multicolumn{2}{c}{$\mathbb{E}$}       & \multicolumn{2}{c}{$\mathbb{V}ar$}   \\
          &       &       & $t=1$ & $t=100$ & $t=1$ & $t=100$ \\
					\midrule
    1.4   & 0.5   & 10    & 0.022 & 0.006 & 0.145& -0.09 \\
          &       & 100   & -0.021 & 0.004 & -0.03 & -0.029 \\
          &       & 1000  & -0.005& -0.008 & -0.032 & 0.066 \\
          &       & 10000 & -0.021 & 0.005& -0.055 & 0.354 \\
          & 1     & 10    & 0.01 & 0.006 & 0.005 & -0.018 \\
          &       & 100   & -0.012 & 0.009 & 0.024 & -0.059 \\
          &       & 1000  & -0.002 & -0.01 & -0.048 & 0.065 \\
          &       & 10000 & -0.02 & -0.005 & -0.0005 & 0.07 \\
          & 2     & 10    & 0.006 & 0.006 & 0.007 & -0.011 \\
          &       & 100   & -0.008 & 0.009 & 0.03 & -0.065 \\
          &       & 1000  & -0.003 & -0.014 & -0.042 & 0.034 \\
          &       & 10000 & -0.022 & -0.011 & -0.005 & 0.058 \\
    1.8   & 0.5   & 10    & 0.009 & 0.017 & 0.12 & -0.099 \\
          &       & 100   & -0.041 & 0.011 & -0.164 & -0.043 \\
          &       & 1000  & 0.013 & -0.016 & -0.07 & 0.287 \\
          &       & 10000 & -0.013 & 0.018 & 0.132 & 0.321 \\
          & 1     & 10    & 0.001 & 0.013 & 0.076 & -0.065 \\
          &       & 100   & -0.038 & 0.014 & -0.149 & -0.162 \\
          &       & 1000  & 0.017 & -0.021 & -0.025 & 0.24 \\
          &       & 10000 & -0.02 & 0.005 & 0.151 & 0.257 \\
          & 2     & 10    & -0.002 & 0.013 & 0.08 & -0.052 \\
          &       & 100   & -0.033 & 0.017 & -0.13 & -0.173 \\
          &       & 1000  & 0.016 & -0.025 & -0.002 & 0.192 \\
          &       & 10000 & -0.023 & -0.005 & 0.147 & 0.22 \\
					\bottomrule
    \end{tabular}%
	  \caption{Accuracy estimations for expectations and variances based on 10000 sample paths of the CARMA-pCTS on $[0,100]$ for different $\alpha$ and $p$. Other parameters fixed at $\delta_+=\delta_-=\lambda_+=\lambda_-=1,\bm a = (3,2)^{\mathrm{T}},\bm b=(3,1)^{\mathrm{T}}$. We use \eqref{eq:CARMA1} for simulation for different $n$ with $\kappa=100.$}
		  \label{tab:pCTS}%
\end{table}%

Similarly to the previous example, we simulate i.i.d.~random variates following the pCTS distribution. We compare the empirical distribution with the theoretical counterpart. Let $\alpha\in\{1.4,1.8\}$ and set $p=1$ and all remaining parameters as above. Table \ref{fig:densCTS} and Figure \ref{tab:CTS} collect the results. We again observe as $\alpha$ increases the goodness-of-fit of the empirical distribution diminishes for fixed $n$. While $n=10000$ seems adequate for $\alpha=1.4$, even $n=100000$ fails to generate satisfactory random variates for $\alpha=1.8$. However, note that there are other simulation approaches for the CTS distribution for $\alpha\in(1,2)$, e.g., \cite{Kawai2011} propose an approximate acceptance-rejection algorithm.

\begin{figure*}
	\centering
\subfloat[][]{\includegraphics[width=.5\textwidth]{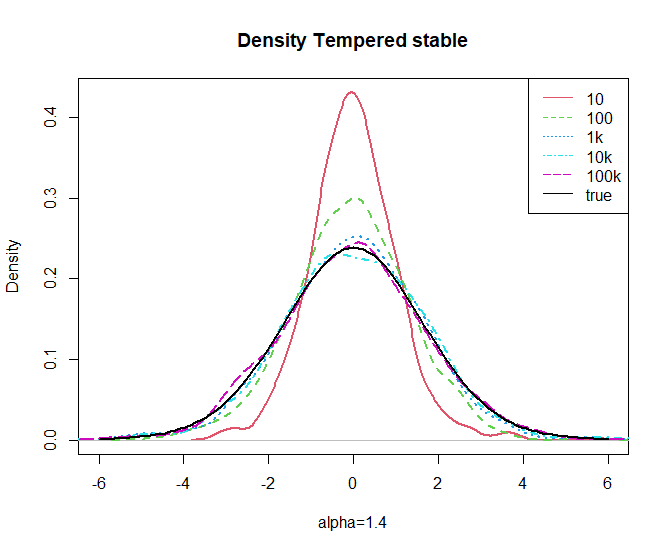}} \hfill
\subfloat[][]{\includegraphics[width=.5\textwidth]{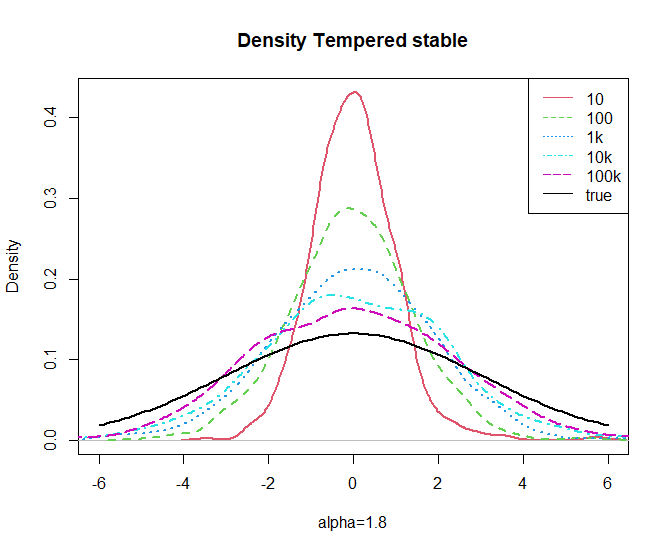}} \hfill
  \caption{Empirical densities of 1000 i.i.d.~realizations of a pCTS random variable for $\alpha\in\{1.4,1.8\}$. Other parameters fixed $p=1,\delta_+=\delta_-=\lambda_+=\lambda_-=1,\bm a = (3,2)^{\mathrm{T}},\bm b=(3,1)^{\mathrm{T}}$. We compare different $n$ and set $\kappa=100$.}
  \label{fig:densCTS}
\end{figure*}

\begin{table}[htbp]
  \centering
    \begin{tabular}{rrrr}
		\toprule
    $\alpha$ & $n$ & $\mathbb{E}$       & $\mathbb{V}ar$   \\
					\midrule
    0.5       & 10    & 0.0005 & -1.882  \\
                & 100   & -0.057 & -1.056 \\
          & 1000  & -0.007 & -0.305 \\
                 & 10000 & 0.047 & -0.03  \\
					& 100000 & -0.032 &  -0.085 \\
    0.8      & 10    & 0.0004 &  -8.221 \\
                 & 100   & -0.054 & -7.15 \\
                & 1000  & -0.004 & -5.71 \\
                 & 10000 & 0.073 & -4.399 \\
								& 100000 & -0.016 & -3.72 \\
					\bottomrule
    \end{tabular}%
	  \caption{Accuracy estimations for expectations and variances based on 1000 i.i.d.~realizations of a pCTS random variable for $\alpha\in\{1.4,1.8\}$. Other parameters fixed $p=1,\delta_+=\delta_-=\lambda_+=\lambda_-=1,\bm a = (3,2)^{\mathrm{T}},\bm b=(3,1)^{\mathrm{T}}$. We compare different $n$ and set $\kappa=100$.}
		  \label{tab:CTS}%
\end{table}%

\end{example}

\begin{example}[continues=expl:GTS]
Lastly, we turn to the case of the CARMA-p$\Gamma$TS process. Again, we set $T=\kappa=100$, $\bm a = (3,2)^{\mathrm{T}}$ and $\bm b=(3,1)^{\mathrm{T}}$ and $p\in\{0.5,1,2\}$. Additionally, we set $\beta=3,\lambda=1$ and $\alpha\in\{0.5,0.8\}$. Given that $||\sigma||=1$, $\tilde{Q}=\Gamma\left(\beta/p,\lambda\right)$ and thus it is a standard task to sample $\{V_j\}$. Because we are dealing with a subordinator we can use \eqref{eq:CARMAsub} to sample paths. Formulas for the expectations in \eqref{eq:CARMAsub} are cumbersome and relegated to Appendix \ref{app:moments}. Figure \ref{fig:CARMAGTS} shows a typical path for $\alpha=0.5,p=1,n=10000$. Table \ref{tab:pGTS} reports the estimated accuracy for expectation and variance. Unlike in the other examples, there appears to be no trend of less accurate variance estimation for small values of $p$ in this instance.

\begin{figure*}
	\centering
  \includegraphics[width=.75\textwidth]{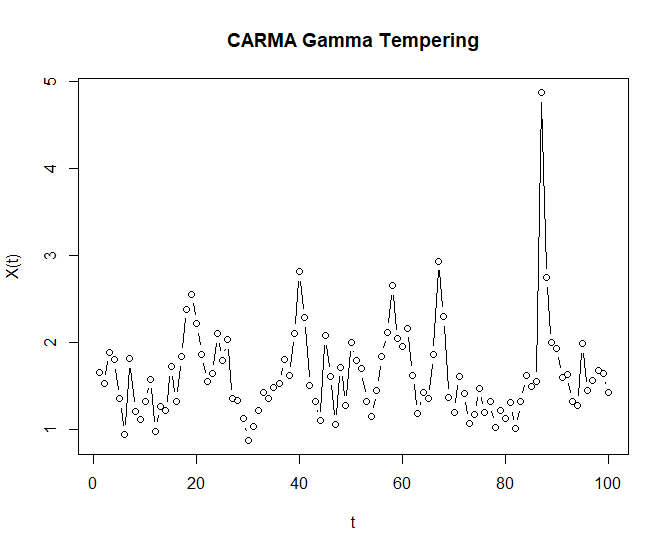}
  \caption{Sample paths of the CARMA-p$\Gamma$TS process on $[0,100]$. Parameters fixed at $\alpha=0.5,p=1,\beta=3,\lambda=1,\bm a = (3,2)^{\mathrm{T}},\bm b=(3,1)^{\mathrm{T}}$. We use \eqref{eq:CARMAsub} for simulation with $n=10000,\kappa=100.$}
  \label{fig:CARMAGTS}
\end{figure*}

\begin{table}[htbp]
  \centering
  \begin{tabular}{rrrrrrr}
		\toprule
    $\alpha$ & $p$& $n$ & \multicolumn{2}{c}{$\mathbb{E}$}       & \multicolumn{2}{c}{$\mathbb{V}ar$}   \\
          &       &       & $t=1$ & $t=100$ & $t=1$ & $t=100$ \\
					\midrule
    0,5   & 0,5   & 10    & 0,003 & -0,001 & 0,017 & 0,001 \\
          &       & 100   & -0,0002 & -0,0003 & -0,005 & -0,007 \\
          &       & 1000  & -0,002 & 0,001 & -0,009 & -0,005 \\
          &       & 10000 & 0,001 & 0,002 & -0,003 & -0,008 \\
          & 1     & 10    & 0,002 & -0,003 & -0,016 & 0,025 \\
          &       & 100   & -0,001 & 0,004 & -0,008 & 0,039 \\
          &       & 1000  & -0,004 & 0,01 & -0,014 & 0,006 \\
          &       & 10000 & 0,009 & 0,001 & 0,02 & -0,019 \\
          & 2     & 10    & 0,011 & 0,006 & 0,01 & 0,022 \\
          &       & 100   & -0,001 & -0,004 & -0,011 & -0,035 \\
          &       & 1000  & -0,012 & 0,01 & -0,015 & -0,02 \\
          &       & 10000 & 0,006 & 0,013 & -0,004 & 0,011 \\
    0,8   & 0,5   & 10    & 0,002 & -0,001 & -0,013 & -0,007 \\
          &       & 100   & 0,001 & 0,001 & -0,002 & -0,002 \\
          &       & 1000  & -0,001 & 0,001 & -0,002 & 0,004 \\
          &       & 10000 & -0.0000 & 0,002 & -0,005 & -0,001 \\
          & 1     & 10    & -0,001 & -0,001 & -0,03 & 0,013 \\
          &       & 100   & 0,001 & 0,004 & 0,003 & 0,05 \\
          &       & 1000  & -0,004 & 0,008 & -0,008 & 0,027 \\
          &       & 10000 & 0,01 & 0,002 & 0,011 & -0,007 \\
          & 2     & 10    & 0,006 & 0,001 & -0,034 & -0,026 \\
          &       & 100   & -0,0004 & -0,002 & 0,003 & -0,022 \\
          &       & 1000  & -0,015 & 0,006 & -0,049 & -0,001 \\
          &       & 10000 & 0,005 & 0,007 & -0,0002 & -0,003 \\
					\bottomrule
    \end{tabular}%
	  \caption{Accuracy estimations for expectations and variances based on 10000 sample paths of the CARMA-p$\Gamma$TS on $[0,100]$ for different $\alpha$ and $p$. Other parameters fixed at $\beta=3,\lambda=1,\bm a = (3,2)^{\mathrm{T}},\bm b=(3,1)^{\mathrm{T}}$. We use \eqref{eq:CARMAsub} for simulation for different $n$ with $\kappa=100.$}
		  \label{tab:pGTS}%
\end{table}%

\end{example}

\section{Conclusion}\label{sec:conclusion}

In this paper, we conducted a comprehensive analysis of simulation techniques for $p$-tempered $\alpha$-stable L\'evy processes, with a specific focus on CARMA processes driven by such L\'evy processes. We have shown a series representation result and propose a simulation scheme for CARMA-pTS processes based on \cite{Kawai2017} and truncated series representations. We found a representation for the error. In a simulation study we observed the efficacy of our proposed scheme in practical applications for a small level of truncation. Remarkably, our simulation method displayed higher accuracy compared to the simulation of the background-driving L\'evy process itself. An interesting direction for future research is the simultaneous estimation of the the parameters $\alpha$ and $p$ to make $TS_{\alpha}^{p}$ distributions more accessible for financial researchers. Further topics also include generalizations of CARMA processes as volatility modulated Volterra processes \cite[]{barndorff2013modelling}.

\section*{Acknowledgements}
Financial support of the German Research Foundation (Deutsche Forschungsgemeinschaft, DFG) via the project 455257011 is gratefully acknowledged.

The author is grateful to two anonymous referees as well as Christoph Hanck for valuable comments which helped to substantially improve this paper. Full responsibility is taken for
all remaining errors.

\bibliography{bibliography}
\bibliographystyle{agsm}

\appendix

\section{Appendix: Proofs}\label{app:proofs}

\begin{proof}[Proof of Theorem \ref{thm:series}]
We extend the proofs of \cite{Rosinski2007} and \cite{Bianchi2011}. As in \cite{Rosinski2001}, it is enough to prove the statements for fixed $t$. 
Define
\begin{equation}
\label{eq:Hfunction}
H(\Gamma_j;(V_j,E_j,U_j)):=\left(\left(\frac{\alpha \Gamma_j}{||\sigma||T}\right)^{-1/\alpha}\wedge\frac{E_j^{1/p}U_j^{1/\alpha}}{||V_j||^{1/p}}\right)\frac{V_j}{||V_j||}.
\end{equation}
We need to show
\begin{equation}
\label{eq:needtoshow}
\int_0^{\infty}\mathbb{P}\left[\mathds{1}_{(0,t]}(T_j)H(s;(V_j,E_j,U_j))\in A\right]=tM(A)
\end{equation}
for every $0\notin A\in\mathcal{B}(\mathbb{R}^d)$. It is enough to verify \eqref{eq:needtoshow} for sets of the form $A=\{x\in\mathbb{R}:||x||>a,\frac{x}{||x||}\in B\}$, where $a>0$ and $B\in\mathcal{B}(\mathbb{S}^{d-1})$. For such $A$ the left-hand-side of \eqref{eq:needtoshow} can be written as
\begin{align}
&\int_0^{\infty}\mathbb{P}\left[T_j\in(0,t]\right]\mathbb{P}\left[\left(\left(\frac{\alpha \Gamma_j}{||\sigma||T}\right)^{-1/\alpha}\wedge\frac{E_j^{1/p}U_j^{1/\alpha}}{||V_j||^{1/p}}\right)\frac{V_j}{||V_j||}\in A\right]\upd s\\
=&\frac{t}{T}\mathbb{E}\left[\int_0^{\infty}\mathds{1}\left(\left(\frac{\alpha \Gamma_j}{||\sigma||T}\right)^{-1/\alpha}>a,E_j^{1/p}U_j^{1/\alpha}>a||V_j||^{1/p},\frac{V_j}{||V_j||}\in B\right)\upd s\right]\\
=&t\alpha^{-1}||\sigma||a^{-\alpha}\mathbb{E}\left[\mathds{1}\left(E_j^{1/p}U_j^{1/\alpha}>a||V_j||^{1/p},\frac{V_j}{||V_j||}\in B\right)\right]\\
=&t\alpha^{-1}a^{-\alpha}\int_B\int_0^{\infty}\mathbb{P}\left[E_j^{1/p}U_j^{1/\alpha}>a||V_j||^{1/p}\right]Q(\upd s|u)\sigma(\upd u)\\
%=&t\alpha^{-1}a^{-\alpha}\int_B\int_0^{\infty}\alpha\int_1^{\infty}\mathrm{e}^{-a^pr^ps}r^{-\alpha-1}\upd rQ(\upd s|u)\sigma(\upd u)\\
=&t\int_B\int_0^{\infty}\int_a^{\infty}\mathrm{e}^{-r^ps}r^{-\alpha-1}\upd rQ(\upd s|u)\sigma(\upd u)\\
=&t\int_B\int_a^{\infty}q(r^p,u)\upd r \sigma(\upd u)\\
=& tM(A),
\end{align}
where the fourth inequality follows by conditioning and integration by substitution,
\begin{align}
\mathbb{P}\left[E_j^{1/p}U_j^{1/\alpha}>a||V_j||^{1/p}\right]&=\int_0^1\mathrm{e}^{-\frac{a^ps}{x^{p/\alpha}}}\upd x
=a^{\alpha}\alpha\int_a^{\infty}\mathrm{e}^{-r^ps}r^{-\alpha-1}\upd r.
\end{align}
This proves \eqref{eq:needtoshow}.

For $\alpha\in(0,1)$ or $Q$ is symmetric and $\alpha\in[1,2)$, the remainder of the proof is exactly as in \cite{Rosinski2007} which proves (i). For (ii), first consider the case $\alpha\in(1,2)$. We need to show
\begin{equation}
\label{eq:remainstoshow}
\sum_{j=1}^{\infty}\frac{t}{T}\left(\frac{\alpha j}{||\sigma||T}\right)^{-1/\alpha}x_0-\frac{t}{T}c_j=tb_T,
\end{equation}
where
\begin{equation}
\label{eq:centeringc}
c_j=\int_{j-1}^j\mathbb{E}\left[\left(\left(\frac{\alpha s}{||\sigma||T}\right)^{-1/\alpha}\wedge\frac{E_j^{1/p}U_j^{1/\alpha}}{||V_j||^{1/p}}\right)\frac{V_j}{||V_j||}\right]\upd s
\end{equation}
and $b_T$ as in \eqref{eq:bT}.

Define as in \cite{Rosinski2007}
\begin{equation}
\label{eq:centeringcprime}
c_j'=\int_{j-1}^j\mathbb{E}\left[\left(\frac{\alpha s}{||\sigma||T}\right)^{-1/\alpha}\frac{V_j}{||V_j||}\right]\upd s=\frac{\alpha^{1-1/\alpha}(||\sigma||T)^{1/\alpha}}{\alpha-1}(j^{1-1/\alpha}-(j-1)^{1-1/\alpha})x_0,
\end{equation}
for $j\ge1$.
The procedure now is analogously to \cite{Rosinski2007} by first showing showing absolute summability of $||c_j'-c_j||$ and then computing $\sum_{j=1}^{\infty}(c_j'-c_j)$. We only have to replace $\mathbb{E}[E_1^{(1-\alpha)}]$ by $\mathbb{E}[E_1^{(1-\alpha)/p}]=\Gamma\left(\frac{1+p-\alpha}{p}\right)$ in each step and use the conversion rule \eqref{eq:Rmeasure} for general $p$.

For $\alpha =1$, we again follow \cite{Rosinski2007} and replace $\mathbb{E}[\log(E_1U_1)]$ by $\mathbb{E}[\log(E_1^{1/p}U_1)]=-1-\gamma/p$ in each step. This completes the proof.

\end{proof}

\begin{proof}[Proof of Theorem \ref{thm:trunc}]
Theorem \ref{thm:series} implies that 
\begin{equation}
\label{eq:measdecomp}
M((x,\infty)B)=\int_{\mathbb{R}^d}\int_0^{\infty}\int_0^1\int_0^{\infty}\mathds{1}_{(x,\infty)}\left(\left(\frac{\alpha s}{||\sigma||}\right)^{-1/\alpha}\wedge\frac{r^{1/p}u^{1/\alpha}}{||v||^{1/p}}\right)\mathds{1}_B\left(\frac{v}{||v||}\right)\upd s \upd u \mathrm{e}^{-r}\upd r\widetilde{Q}(\upd v),
\end{equation}
for $x>0$ and $B\in\mathcal{B}(\mathbb{S}^{d-1})$, where $\widetilde{Q}(\upd v):=Q(\upd v)/||\sigma||$.

As in \cite{Imai20114411},
\begin{align}
&\int_0^n \mathds{1}_{(x,\infty)}\left(\left(\frac{\alpha s}{||\sigma||}\right)^{-1/\alpha}\wedge\frac{r^{1/p}u^{1/\alpha}}{||v||^{1/p}}\right)\\
=&\mathds{1}_{(x,\infty)} \left(\frac{r^{1/p}u^{1/\alpha}}{||v||^{1/p}}\right)\mathrm{Leb}\left(\left\{s\in(0,n)\left(\frac{\alpha s}{||\sigma||}\right)^{-1/\alpha}>x\right\}\right)\\
=&\mathds{1}_{(x,\infty)} \left(\frac{r^{1/p}u^{1/\alpha}}{||v||^{1/p}}\right)\left(n\wedge\frac{||\sigma||}{\alpha}x^{-\alpha}\right).
\end{align}
Due to the truncation, $L^{(n)}$ has the triplet $(0,0,M_{n})$, where
\begin{align}
\label{eq:truncmeasdecomp}
M_{n}((x,\infty)B)&=\int_{\mathbb{R}^d}\int_0^{\infty}\int_0^1\int_0^{n}\mathds{1}_{(x,\infty)}\left(\left(\frac{\alpha s}{||\sigma||}\right)^{-1/\alpha}\wedge\frac{r^{1/p}u^{1/\alpha}}{||v||^{1/p}}\right)\mathds{1}_B\left(\frac{v}{||v||}\right)\upd s \upd u \mathrm{e}^{-r}\upd r\widetilde{Q}(\upd v)\\
&=\int_{\mathbb{R}^d}\int_0^{\infty}\int_0^1\mathds{1}_{(x,\infty)} \left(\frac{r^{1/p}u^{1/\alpha}}{||v||^{1/p}}\right)\left(n\wedge\frac{||\sigma||}{\alpha}x^{-\alpha}\right)\mathds{1}_B\left(\frac{v}{||v||}\right)\upd u \mathrm{e}^{-r}\upd r\widetilde{Q}(\upd v)\\
&=\left(n\frac{\alpha}{||\sigma||}x^{\alpha}\wedge1\right)\int_{\mathbb{R}^d}\int_0^{\infty}\int_0^1\mathds{1}_{(x,\infty)}\left(\frac{r^{1/p}u^{1/\alpha}}{||v||^{1/p}}\right)\frac{||\sigma||}{\alpha}x^{-\alpha}\mathds{1}_B\left(\frac{v}{||v||}\right)\upd u \mathrm{e}^{-r}\upd r\widetilde{Q}(\upd v)\\
&= \left(n\frac{\alpha}{||\sigma||}x^{\alpha}\wedge1\right)\int_{\mathbb{R}^d}\int_0^{\infty}\int_0^1\int_0^{\infty}\mathds{1}_{(x,\infty)}\left(\left(\frac{\alpha s}{||\sigma||}\right)^{-1/\alpha}\wedge\frac{r^{1/p}u^{1/\alpha}}{||v||^{1/p}}\right)\mathds{1}_B\left(\frac{v}{||v||}\right)\upd s \upd u \mathrm{e}^{-r}\upd r\widetilde{Q}(\upd v)\\
&= \left(n\frac{\alpha}{||\sigma||}x^{\alpha}\wedge1\right) M((x,\infty)B).
\end{align}
The fourth inequality follows since, analogously to \cite{Imai20114411},
\begin{equation}
\mathds{1}_{(x,\infty)}\left(\frac{r^{1/p}u^{1/\alpha}}{||v||^{1/p}}\right)\frac{||\sigma||}{\alpha}x^{-\alpha}=\int_0^{\infty}\mathds{1}_{(x,\infty)}\left(\left(\frac{\alpha s}{||\sigma||}\right)^{-1/\alpha}\wedge\frac{r^{1/p}u^{1/\alpha}}{||v||^{1/p}}\right)\upd s.
\end{equation}
\end{proof}

\begin{proof}[Proof of Corollary \ref{cor:error}]
For the first,
\begin{align}
\mathbb{E}\left[(Y_t-\widetilde{Y}_t(\kappa,n))^2\right]&\le\mathbb{E}[Q_t(n)^2]+\mathbb{E}[R_t(\kappa,n)^2]\\%+\sigma_n^2\mathbb{V}ar\left[\int_{-\infty}^Tg(t-s)\upd W_s\right]\\
&=\mathbb{V}ar[Q_t(n)]+\mathbb{E}[Q_t(n)]^2+\mathbb{V}ar[R_t(\kappa,n)]+\mathbb{E}[R_t(\kappa,n)]^2.%+\sigma_n^2\bm b^{\mathrm{T}}\bm \Sigma \bm b,
\end{align}
%because of the variance formula \eqref{eq:covYt} for a Gaussian CARMA process.
We have by the definitions
\begin{align}
\mathbb{E}[Q_t(n)]&=\int_{\mathbb{R}}z(M-M_n)(\upd z)\int_{-\infty}^Tg(t-s)\upd s=-\int_{\mathbb{R}}z(M-M_n)(\upd z)\sum_{k=1}^{\bar{p}}\frac{\alpha_k}{\lambda_k},\\
\mathbb{E}[R_t(\kappa,n)]&=\int_{\mathbb{R}}M_n(\upd z)\int_{-\infty}^{-\kappa}g(t-s)\upd s=-\int_{\mathbb{R}}M_n(\upd z)\sum_{k=1}^{\bar{p}}\frac{\alpha_k\mathrm{e}^{\lambda_k(\kappa+t)}}{\lambda_k},
\end{align}
and
\begin{align}
\mathbb{V}ar[Q_t(n)]&=\int_{\mathbb{R}}z^2(M-M_n)(\upd z)\int_{-\infty}^Tg(t-s)^2\upd s\le -\sigma_n^2\left(\sum_{k=1}^{\bar{p}}\frac{1}{2\lambda_k}\right)\left(\sum_{k=1}^{\bar{p}}\alpha_k^2\right),\\
\mathbb{V}ar[R_t(\kappa,n)]&=\int_{\mathbb{R}}z^2M_n(\upd z)\int_{-\infty}^{-\kappa}g(t-s)^2\upd s\le -\int_{\mathbb{R}}z^2M_n(\upd z)\left(\sum_{k=1}^{\bar{p}}\frac{\mathrm{e}^{2\lambda_k(\kappa+t)}}{2\lambda_k}\right)\left(\sum_{k=1}^{\bar{p}}\alpha_k^2\right),
\end{align}
where we have used the Cauchy-Schwarz inequality. This implies the first result. With this, the second is straightforward. 
\end{proof}

\section{Appendix: L\'evy measure moments}\label{app:moments}
We collect formulas of the truncated L\'evy measure, i.e.,
\begin{equation}
\label{eq:Levymoment1and2}
\int zM_n(\upd z)\ \ \ \text{and}\ \ \ \int z^2M_n(\upd z)
\end{equation}
for our examples. With these it is computationally straight-forward to derive $\mathbb{E}[Q_t(n)]$, $\mathbb{E}[R_t(\kappa,n)]$, $\mathbb{V}ar[Q_t(n)]$, and $\mathbb{V}ar[R_t(\kappa,n)]$.

\begin{example}[continues=expl:TSS]
With the help of \eqref{eq:truncmeasure}, we obtain for the pTSS
\begin{align}
\label{eq:truncLevymeas1TSS}
\int_0^{\infty} zM_n(\upd z)&= \alpha  n  \left(\left(\frac{\alpha  n}{\delta }\right)^{-1/\alpha }\left((\alpha +1)   \mathrm{E}_{\frac{\alpha +p-1}{p}}\left(\left(\frac{n \alpha }{\delta }\right)^{-p/\alpha } \lambda ^p\right)-\alpha   \mathrm{E}_{\frac{\alpha }{p}+1}\left(\left(\frac{n \alpha }{\delta }\right)^{-p/\alpha } \lambda
   ^p\right)\right)\right.\\
	&\ \ \ + \left. \lambda^{-1}\left(\Gamma \left(\frac{1}{p}\right)-\Gamma \left(\frac{1}{p},\left(\frac{n \alpha }{\delta }\right)^{-p/\alpha } \lambda ^p\right)\right)\right)\left(p(\alpha +1)\right)^{-1} ,\\
\label{eq:truncLevymeas2TSS}
\int_0^{\infty} z^2M_n(\upd z)&=\alpha  n \left(\left(\frac{\alpha  n}{\delta }\right)^{-2/\alpha } \left((\alpha +2) \mathrm{E}_{\frac{p+\alpha -2}{p}}\left(\left(\frac{n \alpha }{\delta }\right)^{-\frac{p}{\alpha }}
   \lambda ^p\right)-\alpha  \mathrm{E}_{\frac{p+\alpha }{p}}\left(\left(\frac{n \alpha }{\delta }\right)^{-\frac{p}{\alpha }} \lambda ^p\right)\right)\right.\\
	&\ \ \ +\left.2\lambda^{-2} \left(\Gamma
   \left(\frac{2}{p}\right)-\Gamma \left(\frac{2}{p},\left(\frac{n \alpha }{\delta }\right)^{-\frac{p}{\alpha }} \lambda ^p\right)\right)\right)\left(p (\alpha +2)\right)^{-1},
\end{align}
where $\mathrm{E}_m(x)=\int_1^{\infty}\mathrm{e}^{-xt}t^{-m}\upd t$ is the exponential integral function and $\Gamma(s,x)=\int_x^{\infty}t^{s-1}\mathrm{e}^{-t}\upd t$ is the upper incomplete gamma function.
\end{example}

\begin{example}[continues=expl:CTS]
\begin{align}
\label{eq:truncLevymeas1CTS}
\int_{-\infty}^{\infty} zM_n(\upd z)&= \delta_+ \left(\left(\frac{\alpha  n}{||\sigma|| }\right)^{-1/\alpha +1}\left((\alpha +1)   \mathrm{E}_{\frac{\alpha +p-1}{p}}\left(\left(\frac{n \alpha }{||\sigma|| }\right)^{-p/\alpha } \lambda_+ ^p\right)-\alpha   \mathrm{E}_{\frac{\alpha }{p}+1}\left(\left(\frac{n \alpha }{||\sigma|| }\right)^{-p/\alpha } \lambda_+
   ^p\right)\right)\right.\\
	&\ \ \ + \left. \lambda_+^{-1}\frac{\alpha  n}{||\sigma|| }\left(\Gamma \left(\frac{1}{p}\right)-\Gamma \left(\frac{1}{p},\left(\frac{n \alpha }{||\sigma||}\right)^{-p/\alpha } \lambda_+ ^p\right)\right)\right)\left(p(\alpha +1)\right)^{-1} ,\\
&\ \ \ +	\delta_- \left(\left(\frac{\alpha  n}{||\sigma|| }\right)^{-1/\alpha +1}\left((\alpha +1)   \mathrm{E}_{\frac{\alpha +p-1}{p}}\left(\left(\frac{n \alpha }{||\sigma|| }\right)^{-p/\alpha } \lambda_- ^p\right)-\alpha   \mathrm{E}_{\frac{\alpha }{p}+1}\left(\left(\frac{n \alpha }{||\sigma|| }\right)^{-p/\alpha } \lambda_-
   ^p\right)\right)\right.\\
&\ \ \ + \left. \lambda_-^{-1}\frac{\alpha  n}{||\sigma|| }\left(\Gamma \left(\frac{1}{p}\right)-\Gamma \left(\frac{1}{p},\left(\frac{n \alpha }{||\sigma||}\right)^{-p/\alpha } \lambda_- ^p\right)\right)\right)\left(p(\alpha +1)\right)^{-1} ,\\	
\label{eq:truncLevymeas2CTS}
\int_{-\infty}^{\infty} z^2M_n(\upd z)&=\delta_+ \left(\left(\frac{\alpha  n}{||\sigma|| }\right)^{-2/\alpha +1} \left((\alpha +2) \mathrm{E}_{\frac{p+\alpha -2}{p}}\left(\left(\frac{n \alpha }{||\sigma||}\right)^{-\frac{p}{\alpha }}
   \lambda_+ ^p\right)-\alpha  \mathrm{E}_{\frac{p+\alpha }{p}}\left(\left(\frac{n \alpha }{||\sigma|| }\right)^{-\frac{p}{\alpha }} \lambda_+ ^p\right)\right)\right.\\
	&\ \ \ +\left.2\lambda_+^{-2} \frac{\alpha  n}{||\sigma|| } \left(\Gamma
   \left(\frac{2}{p}\right)-\Gamma \left(\frac{2}{p},\left(\frac{n \alpha }{||\sigma|| }\right)^{-\frac{p}{\alpha }} \lambda_+ ^p\right)\right)\right)\left(p (\alpha +2)\right)^{-1}\\
&\ \ \ + \delta_- \left(\left(\frac{\alpha  n}{||\sigma|| }\right)^{-2/\alpha +1} \left((\alpha +2) \mathrm{E}_{\frac{p+\alpha -2}{p}}\left(\left(\frac{n \alpha }{||\sigma||}\right)^{-\frac{p}{\alpha }}
   \lambda_- ^p\right)-\alpha  \mathrm{E}_{\frac{p+\alpha }{p}}\left(\left(\frac{n \alpha }{||\sigma|| }\right)^{-\frac{p}{\alpha }} \lambda_- ^p\right)\right)\right.\\
	&\ \ \ +\left.2\lambda_-^{-2} \frac{\alpha  n}{||\sigma|| } \left(\Gamma
   \left(\frac{2}{p}\right)-\Gamma \left(\frac{2}{p},\left(\frac{n \alpha }{||\sigma|| }\right)^{-\frac{p}{\alpha }} \lambda_- ^p\right)\right)\right)\left(p (\alpha +2)\right)^{-1},
\end{align}
where in this case $||\sigma||=\delta_++\delta_-$.
\end{example}

\begin{example}[continues=expl:GTS]
\begin{align}
\label{eq:truncLevymeas1GTS}
\int_0^{\infty} zM_n(\upd z)&=\Bigg( (n\alpha)^{1-\alpha-\beta}\lambda^{\beta/p}\left((\alpha+1)(\alpha+\beta-1)^{-1}\ _2F_1\left(\frac{\beta}{p},\frac{\alpha+\beta-1}{p},\frac{\alpha+\beta+p-1}{p};-(n\alpha)^{-p}\lambda\right)\right.\\
&\ \ \ + \left((n\alpha)^{\alpha+1}-\alpha-1\right)(\alpha+\beta)^{-1}\ _2F_1\left(\frac{\beta}{p},\frac{\alpha+\beta}{p},\frac{\alpha+\beta+p}{p};-(n\alpha)^{-p}\lambda\right)\\
&\ \ \ - \left.(n\alpha)^{\alpha+1}(\beta-1)^{-1}\ _2F_1\left(\frac{\beta-1}{p},\frac{\beta}{p},\frac{\beta+p-1}{p};-(n\alpha)^{-p}\lambda\right)\right)\\
&\ \ \ + n\alpha\lambda^{1/p}\Gamma\left(1+\frac{1}{p}\right)\Gamma\left(\frac{\beta-1}{p}\right)\Gamma\left(\frac{\beta}{p}\right)^{-1}\Bigg)(\alpha+1)^{-1},
\end{align}
if $\alpha+\beta>1$ and $\beta\neq1$, where $ _2F_1(a,b,c;x)=\sum_{j=0}^{\infty}\frac{(a)_j(b_j)}{(c)_j}\frac{x^n}{n!}$ is the hypergeometric function and $(q)_j$ denotes the Pochhammer symbol. We omit the formula for $\beta=1$.
\begin{align}
\label{eq:truncLevymeas2GTS}
\int_0^{\infty} z^2M_n(\upd z)&=\Bigg( (n\alpha)^{2-\alpha-\beta}\lambda^{\beta/p}\left((\alpha+2)(\alpha+\beta-2)^{-1}\ _2F_1\left(\frac{\beta}{p},\frac{\alpha+\beta-2}{p},\frac{\alpha+\beta+p-2}{p};-(n\alpha)^{-p}\lambda\right)\right.\\
&\ \ \ + \left(2(n\alpha)^{\alpha+1}-\alpha-2\right)(\alpha+\beta)^{-1}\ _2F_1\left(\frac{\beta}{p},\frac{\alpha+\beta}{p},\frac{\alpha+\beta+p}{p};-(n\alpha)^{-p}\lambda\right)\\
&\ \ \ - \left.2(n\alpha)^{\alpha+1}(\beta-2)^{-1}\ _2F_1\left(\frac{\beta-2}{p},\frac{\beta}{p},\frac{\beta+p-2}{p};-(n\alpha)^{-p}\lambda\right)\right)\\
&\ \ \ + n\alpha\lambda^{2/p}\Gamma\left(1+\frac{2}{p}\right)\Gamma\left(\frac{\beta-2}{p}\right)\Gamma\left(\frac{\beta}{p}\right)^{-1}\Bigg)(\alpha+2)^{-1},
\end{align}
if $\alpha+\beta>2$ and $\beta\neq2$. We omit the formula for $\beta=2$.
\end{example}
\end{document}